\documentclass[11pt,a4paper,makeidx]{amsart}
\usepackage[latin1]{inputenc}        
\usepackage[dvips]{graphics}
\usepackage{graphicx}
\usepackage{wrapfig}

\usepackage{amssymb}
\usepackage{amsmath}
\usepackage{amsfonts}
\usepackage{latexsym}
\usepackage[T1]{fontenc}
\usepackage{hyperref}
\usepackage[all,ps]{xy}
\usepackage{amscd}

\usepackage{epstopdf}
\usepackage{mathrsfs}
\usepackage{color}		
\usepackage{xcolor}	
\usepackage{epsfig}
\usepackage{caption}

\let\mathcal\mathscr
\def\B{{\mathbb B}}
\def\Z{{\mathbb Z}}
\def\N{{\mathbb N}} 
\def\P{{\mathbb P}}

\def\Q{{\mathbb Q}}
\def\C{{\mathbb C}}
\def\ev{\mathop{\rm ev}\nolimits}

\def\codim{\mathop{\rm codim}\nolimits}

\def\Pic{\mathop{\rm Pic}\nolimits}

\def\tilde{\widetilde}

\def\phi{\varphi}

\def\cC{{\mathcal C}}

\def\cH{{\mathcal H}}

\def\cM{{\mathcal M}}
\def\cO{{\mathcal O}}

\def\cX{{\mathcal X}}
\def\cY{{\mathcal Y}}

\def\Pic{\mathop{\rm Pic}\nolimits}
\def\dim{\mathop{\rm dim}\nolimits}
\def\div{\mathop{\rm div}\nolimits}

\newcommand{\Bl}{{\mathrm{Bl}}}
\newcommand{\ra}{\rightarrow}

\theoremstyle{plain}

\newtheorem{thm}{Theorem}[section]

\newtheorem{claim}[thm]{Claim}

\newtheorem{alg}[thm]{Algorithm}

\newtheorem{example}[thm]{Example}

\theoremstyle{definition}

\newtheorem{definition}[thm]{Definition}

\newtheorem{cor}[thm]{Corollary}
\newtheorem{corollary}[thm]{Corollary}
\newtheorem{prop}[thm]{Proposition}
\newtheorem{proposition}[thm]{Proposition}
\newtheorem{lem}[thm]{Lemma}
\newtheorem{lemma}[thm]{Lemma}

\theoremstyle{remark}

\newtheorem{rmk}[thm]{Remark}
\newtheorem{remark}[thm]{Remark}
\newtheorem{setting}[thm]{Setting}

\newcommand{\abs}[1]{{\left|#1\right|}}

\newcommand{\Bs}{{\operatorname{Bs}}}

\newcommand{\ch}{{\rm ch}}

\newcommand{\chern}{{\rm c}}

\newcommand{\Def}{{\operatorname{Def}}}

\newcommand{\Hdg}{\operatorname{Hdg}}

\newcommand{\id}{{\rm id}}

\newcommand{\lt}{{\rm{lt}}}

\newcommand{\NS}{\operatorname{NS}}

\renewcommand{\O}{{\rm O}}

\newcommand{\ol}[1]{{\overline{#1}}}

\newcommand{\ratl}{\dashrightarrow}

\newcommand{\reg}{{\operatorname{reg}}}

\newcommand{\rk}{{\rm rk}}

\newcommand{\sing}{{\operatorname{sing}}}

\renewcommand{\to}[1][]{\xrightarrow{\ #1\ }}

\newcommand{\tensor}{\otimes}

\newcommand{\tf}{{\rm tf}}

\newcommand{\vphi}{\varphi}

\newcommand{\wt}[1]{{\widetilde{#1}}}

\newcommand{\sC}{{\mathcal C}}
\newcommand{\sD}{{\mathcal D}}

\newcommand{\sH}{{\mathcal H}}

\newcommand{\sJ}{{\mathcal J}}

\newcommand{\sM}{{\mathcal M}}

\newcommand{\sS}{{\mathcal S}}

\newcommand{\sX}{{\mathcal X}}
\newcommand{\sY}{{\mathcal Y}}

\newcommand{\gothv}{{\mathfrak v}}

\newcommand{\gothM}{{\mathfrak M}}

\title[Deformations of rational curves on primitive symplectic varieties]{Deformations of rational curves on primitive symplectic varieties and applications}

\author{Christian Lehn}
\address{Christian Lehn\\ Fakult\"at f\"ur Mathematik\\ Technische Universit\"at Chemnitz\\
Reichenhainer Stra\ss e 39, 09126 Chemnitz, Germany}
\email{christian.lehn@mathematik.tu-chemnitz.de}

\author{Giovanni Mongardi}
\address{Giovanni Mongardi\\ Alma Mater studiorum Universit\'a di Bologna
Dipartimento di Matematica,
Piazza di Porta San Donato 5,
Bologna, 40126 Italia}
\email{giovanni.mongardi2@unibo.it}

\author{Gianluca Pacienza}
\address{Gianluca Pacienza\\ Institut Elie Cartan de Lorraine,
Universit\'e de Lorraine,
B.P. 70239, F-54506 Vandoeuvre-l\'es-Nancy Cedex
 France}
\email{gianluca.pacienza@univ-lorraine.fr}

\begin{document}

%
\begin{abstract}
We study the deformation theory of rational curves on primitive symplectic varieties and show that if the rational curves cover a divisor, then, as in the smooth case, they deform along their Hodge locus in the universal locally trivial deformation. As applications, we extend Markman's deformation invariance of prime exceptional divisors along their Hodge locus to this singular framework and provide existence results for uniruled ample divisors on primitive symplectic varieties which are locally trivial deformations of any moduli space of semistable objects on a projective $K3$ or fibers of the Albanese map of those on an abelian surface. We also present an application to the existence of prime exceptional divisors.
\end{abstract}
%
%
%
%

\makeatletter
\@namedef{subjclassname@2020}{
	\textup{2020} Mathematics Subject Classification}
\makeatother







\maketitle
{\let\thefootnote\relax
\footnotetext{
CL was supported by the DFG through the research grants Le 3093/2-2 and  Le 3093/3-1. GM is part of INdAM research group ``GNSAGA'' and was partially supported by it. GP was partially supported by the Projet ANR-16-CE40-0008  ``Foliage''. 
}}
\numberwithin{equation}{section}

%
\section{Introduction}
%

In recent years, ``singular'' symplectic varieties have attracted increasing attention. 
On the one hand, this is certainly due to the fact that \emph{irreducible} symplectic varieties appear as factors of the singular Beauville--Bogomolov decomposition for mildly singular K\"ahler spaces with trivial first Chern class, see \cite{GKP16,Dru18, GGK19,HP19,Ca20,BGL20}. On the other hand, in many fundamental aspects of the theory (deformation theory, lattice structure, projectivity, Torelli theorems etc.) even the larger class of \emph{primitive} symplectic varieties  behaves very similarly to their smooth analogs, see \cite{BL16, BL18, Men20}. It is worthwhile pointing out that both these classes of symplectic varieties coincide in the smooth case by a recent theorem of Schwald \cite{Sch20}, see Section~\ref{section symplectic varieties} for precise definitions and a more detailed discussion.

Moduli spaces of (semi)stable sheaves (with respect to a generic polarization) on $K3$ or abelian surfaces play a prominent r\^ole in the theory of {\it smooth} irreducible symplectic varieties, as all known examples of such varieties arise as deformations of moduli space of sheaves on those surfaces or of (crepant) desingularizations thereof. 
Recently, Perego and Rapagnetta showed in \cite{PR18} that ``singular'' moduli spaces of stable sheaves (with respect to a generic polarization) on projective $K3$ or abelian surfaces are irreducible symplectic varieties (see below for more details). Together with quotients by finite groups of symplectic automorphisms (see \cite[Proposition 2.4]{Bea00}), ``singular'' moduli space of stable sheaves then provide a vast source of examples of singular symplectic varieties.

A recent and very active line of research concerns rational curves on irreducible symplectic varieties, either in relation to their birational geometry (see e.g. \cite{BM14mmp, BM14projectivity, BHT15, AV15, AV17, AV19}) or to study their Chow groups (see e.g. \cite{Vo16, CMP19, MP, MPcorr, LP15, zal, shen2017derived, shen2017k3}). So far, these investigations have mainly been carried out in the smooth case and the purpose of the present paper is to initiate the systematic study of rational curves on singular symplectic varieties, a task which seems both natural and relevant.

There are three different main contexts giving rise to rational curves on symplectic varieties: degenerate fibers of Lagrangian fibrations, prime exceptional divisors, uniruled members of ample (or more generally movable) linear systems. These correspond to divisors of zero, negative, respectively positive square with respect to the Beauville--Bogomolov--Fujiki form, see below. It is also in complete analogy to what happens on $K3$ surfaces where the aforementioned contexts correspond to elliptic fibrations (obtained from square zero curves), to $(-2)$-curves, or generalize the Bogomolov--Mumford theorem.

While the definition of irreducible symplectic varieties is more involved (see Definition~\ref{def:ISV}), the geometric features of \emph{primitive} symplectic varieties are mainly their vanishing irregularity and the uniqueness of the symplectic form up to scalars (see Definition~\ref{def:PSV}). Our applications concern rational curves with non-zero square on primitive symplectic varieties. 
The deformation theory of rational curves is one of the key tools to deduce general results from special cases.   The main technical contribution of the present note extends to the most general singular, not necessarily projective setting a result established in \cite{AV15, CMP19} in the smooth case. 
   
\begin{thm}\label{thm:uniruleddivisor}
Let $X$  be a  primitive symplectic variety of dimension $2n$ and $f:C\to X$  a genus zero stable map. Let $M$ be an irreducible component of the space $\overline{{ M_0}}( X, f_*[C])$ of genus zero stable maps containing $[f]$ and suppose that the deformations of $f$ parametrized by $M$ cover a divisor $D\subset X$.
Then the following hold:
\begin{enumerate}
\item There is a unique irreducible component $\mathcal M$ of  the space $\overline{\mathcal{ M}_0}(\mathcal X/S, f_*[C])$ of relative genus zero stable maps in the local universal family $\mathcal X\to S:= \Def^\lt(X)$ of locally trivial deformations of $X$ that contains $M$. Moreover, $\mathcal M$ is smooth at the general point of $M$ and dominates the Hodge locus $B\subset S$ where the class $f_*[C]$ remains algebraic. 
\item For any point $b$ of $B$, the fiber $\mathcal X_b$ contains a uniruled divisor $D_b$ covered by the deformations of $f$ in $\mathcal X_b$. If furthermore $D_b$ is $\Q$-Cartier, then
its cohomology class is proportional to the dual of $f_*[C]$. 
\end{enumerate}
\end{thm}
  
 Here, the dual of a class in $H^2(X,\Q)$ is defined with respect to the Beauville--Bogomolov--Fujiki quadratic form (BBF form), see Definition~\ref{definition dual class}. 
The key (and at first sight maybe somewhat surprising) point in the proof of the above result can most handily be formulated for a primitive symplectic variety $X$ with \emph{terminal singularities}: in this case, the general rational curve ruling a divisor does not meet the singular locus of $X$ (cf. Corollary~\ref{corollary term}). In the projective case, it is standard to reduce to the terminal case via a $\Q$-factorial terminalization, whose existence is not known in the non-projective case. Instead, we use the  functorial resolution of Bierstone--Milman and Villamyor \cite{BM97,Vil89} 
in families and the local structure of symplectic singularities. 

 Theorem \ref{thm:uniruleddivisor} allows us first to generalize to the singular setup the following result concerning prime exceptional divisors (i.e. prime $\Q$-Cartier divisors whose square with respect to the BBF form is negative), which is due to Markman \cite{Mar13} in the smooth case. 
   
\begin{thm}\label{thm:exc-def}
Let $X$ be a projective and $\Q$-factorial primitive symplectic variety. 
Let $E\subset X$ be a prime {exceptional} divisor. Then the following holds.
\begin{enumerate}
\item[(1)] The divisor $E$ is contractible on a birational  $\Q$-factorial primitive symplectic variety $X'$ which is a locally trivial deformation of $X$. In particular, $E$ is uniruled and the contraction of its strict transform on $X'$ determines a distinguished ruling. The dual $E^\vee$ is proportional to the class of a general curve $R$ of this ruling, which is either a smooth rational curve or a union of two smooth rational curves meeting transversally in a single point. Moreover, either $\frac{1}{2}[E]$ or $[E]$ is a primitive class of $H^2(X,\Z)$.
\item[(2)] 
\begin{itemize}
\item[(a)] There is  a flat family of divisors over the Hodge locus $\Hdg_{[E]}(X)$ which specializes to a multiple of $E$ at the origin. In particular, $[E]$ is $\Q$-effective over all its Hodge locus. 
\item[(b)] There exists a non-empty open subset of the Hodge locus $\Hdg_{[E]}(X)$ over which $E$ deforms to a prime exceptional divisor. 
\end{itemize}
\end{enumerate}
\end{thm}
We refer to Definition~\ref{definition uniruled} for the precise notion of a ruling and to Remark~\ref{rmk:uniruled-np} for a discussion about uniqueness. Notice that item (1) does not require to control the deformation theory of the rational curves in the ruling and is essentially contained in \cite[Theorem A and its proof]{BBP13} and \cite[Theorem 3.3]{Dru11}. The above result is one of the ingredients involved in the proof of the fact that reflections in prime exceptional divisors are integral monodromy operators. We refer the reader to \cite{Mar11, Mar13} for details and applications. We plan to return to this topic in subsequent work. 

Moreover, we provide existence results for positive and negative uniruled divisors on primitive symplectic varieties deformation equivalent to a moduli space $M_v(S,\sigma)$  (respectively to a fiber $K_v(S,\sigma)$ of its Albanese map) of semi-stable objects on a projective $K3$ surface $S$ of Mukai vector $v$ which are Bridgeland $\sigma$-semistable with respect to a $v$-generic stability condition (respectively $K_v(S,\sigma)$) on a projective $K3$ (respectively abelian) surface $S$. 
 We refer the reader to Section \ref{ss:mod} for the precise definitions. It is well-known that, when smooth, these moduli spaces are irreducible symplectic manifolds of $K3^{[n]}$ (respectively of $Kum_n$) deformation type. In the singular case, by considering moduli spaces of sheaves of Mukai vector $v$  which are Gieseker $H$-semistable with respect to a $v$-generic polarization (on a projective $K3$, respectively abelian surface $S$)  admitting a crepant resolution, O'Grady discovered the OG10 and OG6 deformation types, cf. \cite{OG99,OG03}. 
In the remaining cases, by the recent results of \cite{PR18}, these moduli spaces are primitive (mostly irreducible) symplectic varieties (again, we refer the reader to Section \ref{ss:mod} for more details and the long history of contributions). 

 Ample uniruled divisors on irreducible symplectic manifolds of $K3^{[n]}$, $Kum_n$ or OG10 deformation types are investigated respectively in \cite{CMP19, MP, B21}. The OG6 deformation type is the object of an ongoing project by Bertini, Grossi, and Onorati. Here we use Theorem \ref{thm:uniruleddivisor} to  show the following. 

\begin{thm}\label{thm:amp}
Let $\mathfrak M$  be any moduli space (cf. Definition~\ref{definition polarized moduli}) of polarized primitive symplectic varieties locally trivially deformation equivalent (as unpolarized varieties) to a fixed moduli space $M_v(S,\sigma)$  of semistable objects on a projective $K3$ surface $S$ of Mukai vector $v$ which are Bridgeland $\sigma$-semistable with respect to a $v$-generic stability condition (or to a fibre $K_v(S,\sigma)$ of its Albanese morphism if $S$ is an abelian surface). 
Then $\mathfrak M$ possesses infinitely many connected components whose points correspond to polarized primitive symplectic varieties all containing an ample uniruled divisor proportional to the polarization.
\end{thm}
The proof of the theorem above uses crucially  the analogous existence results proved in \cite{CMP19, MP} together with a rational map constructed in \cite[Lemma 3.19]{PR18} from a smooth moduli space \emph{of sheaves} $M_u(S,H)$ (respectively $K_u(S,H)$), where $u$ is primitive, onto $M_v(S,H)$ (respectively onto $K_v(S,H)$),  see \eqref{eq:PR}, and another important result contained in \cite{PR18} (cf. Theorem~1.19 therein). Notice that $\mathfrak M$  is of course not of finite type and the generality of our results comes at the price of not knowing whether the connected components of $\gothM$ possibly left out by our results are in finite number or not. Unlike in the smooth case (see \cite{CMP19, MP, B21}), we do not control here the numerical invariants of the connected components covered by our results. It is conceivable however, that our methods could be pushed further by studying the classes of divisors obtained via the Perego--Rapagnetta map \cite[Lemma 3.19]{PR18}. This together with a better knowledge on the monodromy groups in the singular case would allow us to obtain more precise statements, see Remark \ref{rmk:monodromy}.

In terms of marked moduli spaces (cf. Definition~\ref{definition marked moduli space}), Theorem~\ref{thm:amp} can be expressed by saying that every connected component of such a moduli space containing a variety of the form $M_v(S,\sigma)$, respectively $K_v(S,\sigma)$, contains infinitely many divisors parametrizing varieties with a uniruled divisor of positive square. Concerning the existence of exceptional divisors, we prove the following slightly weaker statement. 

\begin{thm}\label{thm:exc}
Let $\cM_\Lambda$  be any moduli space of marked primitive symplectic varieties locally trivially deformation equivalent to a moduli space $M_v(S,\sigma)$  of stable objects on a projective $K3$ surface $S$ of Mukai vector $v$ which are Bridgeland $\sigma$-semistable with respect to a $v$-generic stability condition.
Then any connected component of $\cM_\Lambda$ 
possesses a divisor whose points correspond to polarized primitive symplectic varieties containing a uniruled prime exceptional divisor.
\end{thm}

Since the K\"ahler variety parametrized by the very general point of $\cM_\Lambda$ has Picard group of rank zero, the locus where prime exceptional divisors exist is a priori a countable union of subvarieties of positive codimension. 

{\it Acknowledgments.} We are grateful to G. Ancona, K. Hulek, E. Macr\`i, A. Rapagnetta and K. Yoshioka for their useful comments.

%
\section{Preliminaries}\label{section preliminaries}
%

We will first introduce the various notions of singular symplectic varieties and discuss their basic properties in \ref{section symplectic varieties}, recall the Bogomolov--Beauville--Fujiki form in \ref{section bbf}, and then introduce their moduli spaces in \ref{section symplectic moduli} and \ref{section moduli of polarized symplectic}. In \ref{ss:mod}, we adapt a result due to Perego and Rapagnetta (to the framework of moduli spaces of stable objects) and in \ref{ss:PR}, we present a construction of theirs which is central for the proof of Theorem~\ref{thm:amp}. We conclude the section by showing that moduli spaces of sheaves with respect to a $v$-generic polarization on a projective $K3$ or abelian surface are primitive symplectic varieties. 

\subsection{Symplectic varieties}\label{section symplectic varieties}

Let $X$ be a normal complex variety.
Recall that, for any integer $p\geq 1$, the sheaf $\Omega_X^{[p]}$ of {\it reflexive holomorphic $p$-forms} on $X$
is $\iota_* \Omega_{X_{\reg}}^{p}$, where 
$$\iota : X_{\reg}\hookrightarrow X$$
is the inclusion of the regular locus of $X$. 
It can be alternatively (and equivalently) defined by the double dual $\Omega_X^{[p]}=(\Omega_X^{p})^{**}$.
Recall the following definition which is due to Beauville \cite{Bea00}. 
\begin{definition}\label{def:symp-sing}
Let X be a normal 
variety.
\begin{enumerate}
\item[i)] A \emph{symplectic form} on X is a closed reflexive $2$-form $\sigma$ on $X$
which is non-degenerate at each point of $X_{\reg}$.
\item[ii)]  If $\sigma$ is a symplectic form on $X$, the variety $X$ has \emph{symplectic
singularities} if for one (hence for every) resolution $f : \tilde X \to X$ of the singularities
of $X$, the holomorphic symplectic form $\sigma_{\reg} :=\sigma_{|X_{\reg}}$ extends to
a holomorphic $2$-form on $\tilde X$. In this case, the pair $(X,\sigma)$ is called \emph{symplectic variety}.
\end{enumerate}
\end{definition}

The local structure of symplectic singularities is described by the following proposition, which follows from a combination of \cite[Theorem 2.3]{Kal06}, \cite[Corollary 2.6]{Art69} and \cite[Lemma 1.3]{Nam11}. We refer to \cite[Theorem~3.4]{BL18} for further details regarding the validity of this statement in the complex analytic situation. The proposition will play a role in the proof of Theorem~\ref{thm:exc-def}.
\begin{prop}\label{prop:kaledin}
Let $(X,\sigma)$ be a symplectic variety and let $\Sigma \subset X$ be the singular locus {of $X_{\sing}$}. Then $\codim_X \Sigma \geq 4$ and every $x\in U:=X\setminus \Sigma$  has a neighborhood which is locally analytically isomorphic to $(\C^{2n-2},0)\times (S,p)$ where $2n=\dim X$ and $(S,p)$ is the germ of {a smooth point or} a rational double point on a surface. This isomorphism can be chosen to preserve the symplectic structure.\qed
\end{prop}

\begin{definition}\label{def:PSV}
A \emph{primitive symplectic variety} is a normal compact K\"ahler variety $X$ such that $h^1(X,\mathcal O_X)=0$ and $H^0(X,\Omega_X^{[2]})$ is generated by a holomorphic symplectic form $\sigma$ such that $X$ has symplectic singularities.
\end{definition}

For a normal variety $X$ such that $X_{\reg}$ has a symplectic form $\sigma$, Beauville's condition above that the pullback of $\sigma$ to a resolution of $X$ extends as a regular $2$-form is in fact equivalent to having canonical, even rational singularities by \cite{Elk81}, \cite[Corollary 1.8]{KS18}. For the definition and basic properties of K\"ahler forms on possibly singular complex spaces we refer the reader to e.g. \cite[Section 2]{BL18}.

In order to put Definition \ref{def:PSV} into perspective, recall first the following.
\begin{definition}\label{def:ISV}
An \emph{irreducible symplectic variety} is a normal  compact K\"ahler variety $X$ with canonical singularities
and such that for any  finite morphism $f:X'\to X$ which is \'etale in codimension one the reflexive pull-back $f^{[*]}\sigma$ of the symplectic form $\sigma$ on $X$
generates the exterior algebra of reflexive forms on $X'$. 
\end{definition}
Irreducible symplectic varieties appear in the Beauville--Bogomolov decomposition for numerically $K$-trivial log terminal K\"ahler varieties obtained in \cite[Theorem~A]{BGL20}, building on earlier work in the projective case of \cite{HP19, Dru18, GGK19, GKP16}. In the smooth case, being primitive symplectic or irreducible symplectic is equivalent by \cite[Theorem~1]{Sch20} while in the singular case an irreducible symplectic variety is primitive symplectic but not vice versa. 

To see that irreducibility implies primitivity, it is sufficient to check that $h^1(\mathcal O_X)=0$ whenever $X$ is irreducible symplectic. If $Y\to X$ is any desingularization, then since symplectic singularities are rational we have
$h^1(X,\mathcal O_X)=h^1(Y,\mathcal O_Y)$. The latter is equal to $h^0(Y,\Omega^1_Y)$, which is $\leq h^0(X,\Omega^{[1]}_X)$  (it is actually equal to $h^0(X,\Omega^{[1]}_X)$ by \cite{KS18}, but we do not need this deep result here). To conclude, notice that $ h^0(X,\Omega^{[1]}_X)=0$ by the definition of irreducible symplectic varieties. As counterexample to the converse, one may take for instance the Kummer singular surface $A_{/\pm1}$ which is primitive symplectic but has a cover by the abelian surface $A$ which is finite and \'etale in codimension one, hence it is not irreducible symplectic.  
For a review of the different notions of ``singular'' symplectic varieties we refer the reader to the survey \cite{Pe20}.

As a consequence of the results of \cite{BGL20}, a locally trivial deformation of an irreducible symplectic variety is irreducible symplectic. We sketch the argument.

\begin{lemma}\label{lemma irreducible symplectic locally trivial}
Let $X$ be an irreducible symplectic variety and let $\sX \to \Def^\lt(X)$ be its universal locally trivial deformation. Then $\sX_t$ is irreducible symplectic for any $t\in \Def^\lt(X)$, possibly after shrinking the representative.
\end{lemma}
\begin{proof}
By \cite[Corollary~2.24]{BGL20}, the global sections of the algebra of reflexive differential forms is invariant under locally trivial deformations. As any small deformation of a symplectic form remains symplectic, also on nearby fibers the algebra of holomorphic forms is generated by a nonzero $2$-form. Quasi-\'etale covers of $X$ give rise to quasi-\'etale covers of $\sX$ that prolong the one on $X$ by \cite[Lemma~3.7]{BGL20} and finally quasi-\'etale covers of $\sX$ are again locally trivial by \cite[Lemma~3.8]{BGL20}. The claim now follows from invoking once more the invariance of reflexive forms under locally trivial deformations.
\end{proof}

\subsection{Beauville--Bogomolov--Fujiki form}\label{section bbf}

Let $X$ be a primitive symplectic variety. Then there is a quadratic form $q_X$ on $H^2(X,\C)$, the so-called Beauville--Bogomolov--Fujiki form (\emph{BBF form} for short), see \cite[Definition~5.4]{BL18}. Up to scaling it is defined by the formula
\begin{equation}\label{eq bbf}
q_X(\alpha):= \frac{n}{2} \int_X \left(\sigma\bar\sigma\right)^{n-1}\alpha^2 + (1-n) \int_X \sigma^n\bar\sigma^{n-1}\alpha \int_X \sigma^{n-1}\bar\sigma^n\alpha.
\end{equation}
Due to \cite{Nam01b, Kir15, Mat15, Sch17, BL18} we know that $q_X$ is defined over $\Z$ and nondegenerate of signature $(3,b_2(X)-3)$, see Section~5 of \cite{BL18} and references therein. It is used to formulate the local Torelli theorem \cite[Proposition~5.5]{BL18}, satisfies the Fujiki relations \cite[Proposition~5.15]{BL18}, and is compatible with the Hodge structure on $H^2(X,\Z)$, see \ref{section symplectic moduli} below where this last property is crucial in the definition of the period map.  Moreover, it allows to identify second degree homology with cohomology, more precisely:

\begin{definition}\label{definition dual class}
Let $X$ be primitive symplectic. For $\alpha \in H^2(X,\Q)$, we define the \emph{dual class} $\alpha^\vee\in H^2(X,\Q)^\vee = H_2(X,\Q)$ by the condition
\[
q_X(\alpha,\beta) = \alpha^\vee(\beta).
\]
In the same way, we define $\gamma^\vee\in H^2(X,\Q)$ for a homology class $\gamma\in H_2(X,\Q)$. Clearly, $\alpha^{\vee\vee} = \alpha$.
\end{definition}

\subsection{Moduli spaces of marked primitive symplectic varieties}\label{section symplectic moduli}

Let $\Lambda$ be a lattice of signature $(3,n)$. For a primitive symplectic variety $X$, let $H^2(X,\Z)_\tf$ denote the torsion-free part of its second cohomology.  A $\Lambda$-marking of $X$ is an isometry $\mu : H^2(X,\Z)_\tf \to \Lambda$.  A \emph{$\Lambda$-marked primitive symplectic variety} is a primitive symplectic variety together with the choice of a $\Lambda$-marking. An isomorphism of $\Lambda$-marked primitive symplectic varieties $(X,\mu), (X',\mu')$ is an isomorphism $\vphi:X \to X'$ such that $\mu'=\mu \circ \vphi^*$.

\begin{definition}\label{definition marked moduli space}
We define $\sM_\Lambda$ to be the moduli space of $\Lambda$-marked primitive symplectic varieties of a fixed locally trivial deformation type, that is, its elements are isomorphism classes of $\Lambda$-marked primitive symplectic varieties $(X, \mu)$ where $X$ is a locally trivial deformation of a fixed primitive symplectic variety $X_0$, and the (not-necessarily-Hausdorff topology and) complex structure is obtained from patching Kuranishi spaces for locally trivial deformations together using that miniversal locally trivial deformations are universal, see \cite[Lemma~4.9]{BL18}. From unobstructedness of locally trivial deformations \cite[Theorem~4.7]{BL18}, one deduces that $\sM_\Lambda$ is smooth.
\end{definition}

Let $H=(H_\Z,F^\bullet,q)$ be a \emph{semi-polarized hyperk\"ahler Hodge structure} in the sense of \cite[Definition~8.1]{BL18}, i.e. $H_\Z$ is a $\Z$-module of finite type, $F^\bullet$ is a decreasing filtration on $H_\C:=H_Z\tensor \C$ such that $(H_\Z,F^\bullet)$ becomes a pure Hodge structure of weight $2$, and $q:H_\Z \to \Z$ is a non-degenerate bilinear form of signature $(3,\rk H_\Z -3)$ such that $H^{2,0} \subset H_C$ is an isotropic line and the real space underlying $H^{2,0}\oplus H^{0,2}$ is positive and perpendicular to $H^{1,1}$. As usual, we denoted $H^{p,q}:=F^p\cap \ol F^q$. Let us fix a lattice $\Lambda$ of signature $(3,n)$. Then the period domain for Hodge structures of hyperk\"ahler type on the lattice $\Lambda$ is
\begin{equation}\label{eq period domain}
\sD_\Lambda:=\{ x \in \P(\Lambda\tensor \C) \mid (x,x)=0,\, (x,\bar x) > 0\}
\end{equation}
where $\P(\Lambda\tensor \C)$ denotes the projective space of lines in $\Lambda\tensor \C$ and $x$ is to be interpreted as the $H^{2,0}$-part of the Hodge decomposition, from which the Hodge structure can be reproduced. It is easily seen that $\sD_\Lambda$ is connected. For any sublattice $\Lambda_0 \subset \Lambda$ of signature $(2,n)$, the set $\sD_\Lambda \cap \P(\Lambda_0 \tensor \C)$ has two connected components. We choose one of them and denote it by $\sD_{\Lambda_0}$. In fact, such a choice is determined by the choice of a component of the cone of the positive classes in $\Lambda_0^\perp$, and we can choose it in a coherent way for all such hyperplanes, by taking the component intersecting a fixed component of the positive cone of $\Lambda$.

The period map for $\Lambda$-marked primitive symplectic varieties is the map
\begin{equation}\label{eq period map}
\wp:\sM_\Lambda \to \sD_\Lambda, \qquad (X,\mu) \mapsto \mu(H^{2,0}(X)).
\end{equation}
It is holomorphic and a local isomorphism by the local Torelli theorem \cite[Corollory~5.5]{BL18}. Moreover, it is generically injective when restricted to a connected component  if $\rk \Lambda \geq 5$ by \cite[Theorem~1.1]{BL18} or if $X$ has only quotient singularities by \cite[Theorem~1.1]{Men20}.

\subsection{Moduli spaces of polarized primitive symplectic varieties}\label{section moduli of polarized symplectic}
Recall that a \emph{polarized variety} is a pair $(X,L)$ consisting of a projective variety $X$ and a primitive ample line bundle $L$ on $X$. 
\begin{definition}\label{definition polarized moduli}
For each $d\in \N$ and each locally trivial deformation type $\delta$ of primitive symplectic varieties, we denote $\gothM_{\delta,d}$ the coarse moduli space of polarized primitive symplectic varieties $(X,L)$ where $X$ is a primitive symplectic variety of type $\delta$ and $L$ is an ample line bundle of BBF square $d$. We write 
\[
\gothM_\delta:= \coprod_{d \in \N} \gothM_{\delta,d}.
\]
\end{definition}

Note that $\gothM_{\delta,d}$ exists and is a quasi-projective scheme by \cite[Proposition~8.7 and Lemma~8.8]{BL18} or \cite[Corollary~1.2]{BBT19}. For $b_2(X) \geq 5$, the global Torelli theorem \cite[Theorem~1.1]{BL18} allows to describe the moduli spaces of polarized primitive symplectic varieties. Even though we will not use it, we include it for convenience. The proof goes as in the smooth case, using results from \cite{BL18} instead of their smooth analogs.

\begin{proposition}\label{proposition moduli space polarized}
Let $d\in \N$ and $\delta$ be a locally trivial deformation type of primitive symplectic varieties and suppose that $\rk\Lambda \geq 5$. If $\gothM_0 \subset \gothM_{\delta,d}$ is an irreducible component, then there is a primitive vector $v \in \Lambda$ of square $d$ such that $\gothM_0$ is a Zariski open subset of $\Gamma \backslash \sD_{v^\perp}$ where $\Gamma \subset \O(\Lambda)$ is the subgroup leaving the period domain fixed.\qed
\end{proposition}

\subsection{Moduli spaces of stable objects on K3 or abelian surfaces.}\label{ss:mod}

We recall here some relevant definitions and results about moduli spaces of Bridgeland stable objects on K3 or abelian surfaces. 
These are a generalization of moduli spaces of stable (twisted) sheaves on the same class of (twisted) surfaces. Their construction is based on the existence of a Bridgeland stability condition $\sigma$, proved in \cite[Theorem 1.1]{bridg}. In the following, whenever we use a stability condition we suppose it lies in the component $\mathrm{Stab}^\dagger(S)$ constructed by Bridgeland.
 
Let $S$ be a projective $K3$ surface or an abelian surface. Set 
$$\tilde H(S,\Z):=H^{2\star} (S,\Z)=\Z\oplus H^2 (S,\Z) \oplus\Z. $$ 
An element 
$v=(v_0,v_1,v_2)\in \tilde H(S,\Z)$
is called \emph{Mukai vector} if $v_0\geq 0$, $v_1\in \NS(S)$, and if in case $v_0=0$ either $v_1$ is the first Chern class of an effective divisor, or $v_1=0$  and $v_2 > 0$.
Recall that $\tilde H(S,\Z)$ carries a pure Hodge structure of weight two and is endowed with a lattice structure with respect to a pairing $(\cdot,\cdot)$ called 
\emph{Mukai pairing} (see e.g. \cite[Section~6.1]{HL10} for more details). We set $v^2=(v,v)$ for every Mukai vector  $v\in \tilde H(S,\Z)$ and call  $\left(\tilde H(S,\Z), (\cdot,\cdot)\right)$ the \emph{Mukai lattice} of $S$. 
To any object $\mathcal F^\bullet$ on $D^b(S)$ we associate a Mukai vector $v(\mathcal F^\bullet)$ as follows:
$$
v(\mathcal F^\bullet):= \ch(\mathcal F^\bullet)\sqrt{td(S)}.$$
For an ordinary sheaf $\mathcal F$, we simply have 
$$v(\mathcal F)=\left(\rk(\mathcal F), \chern_1(\mathcal F), \ch_2(\mathcal F) +\epsilon(S)\rk(\mathcal F)\right)
$$
where $\epsilon(S)$ equals $1$ in the $K3$ case and $0$ in the abelian case. 
We consider the moduli space $M_v(S,\sigma)$ of stable objects on $D^b(S)$ of Mukai vector $v$ which are $\sigma$-semistable with respect to a $v$-generic Bridgeland stability condition $\sigma$. In the abelian case we denote by $K_v(S,\sigma)$ the fiber over zero of the Albanese morphism. When no confusion is possible, we will drop the dependence on $S$ and $\sigma$ and simply write $M_v$ and $K_v$.
Notice that, by \cite[Theorem 6.7]{BM14projectivity}, there are $v$-generic stability conditions for which these moduli spaces of Bridgeland-stable objects are ordinary moduli spaces of Gieseker-semistable (twisted) sheaves.
Therefore, basically thanks to the work of Mukai \cite{Muk84}, it is known that the locus $M^s_v$ parametrizing {\it stable} objects is smooth (of dimension $v^2+2\epsilon(S)$) and endowed with a holomorphic symplectic form whenever it is non-empty. For {\it primitive} Mukai vectors the latter occurs precisely when $v^2\geq -2\epsilon(S)$ (see \cite{Yo-k3, Yo-ab}). More classically,  the moduli spaces $M_v(S,H)$ of sheaves on $S$ of Mukai vector $v$ which are Gieseker $H$-semistable with respect to a $v$-generic polarization  (respectively the fiber $K_v(S,H)$ over zero of the Albanese morphism in the abelian case) were considered.  Write $v=mw$, with $w$ primitive Mukai vector and $m\in \N^*$. When $m=1$ (and $w^2>2$ in the abelian case), thanks to results due Huybrechts, Mukai, O'Grady and Yoshioka (see e.g. \cite{Yo-ab} and references therein) we know that  $M_v$ and $K_v$ are irreducible symplectic manifolds deformations deformation equivalent to the punctual Hilbert scheme on a $K3$ surface (respectively to a generalized Kummer variety). If  $m=2$ and $w^2=2$, they possess a symplectic resolution which is an irreducible symplectic manifold (respectively known as the OG10 and  the OG6 manifolds, \cite{OG99}, \cite{OG03}). When $m\geq 3$ or $m=2$ and $w^2\geq 4$, the irreducible symplectic varieties $M_v$ and $K_v$ have  singularities in codimension at least 4 by \cite[Proposition 6.1]{KLS06}, hence they are terminal by \cite[Corollary 1]{Nam01}.

We record below some results that will be crucially used in what follows.

\begin{prop}\label{prop:terminal}
Let $S$ be a projective $K3$ surface or an abelian surface, $v=mw$ a  Mukai vector where $w$ is primitive with $w^2=2k$ for two non-zero integers $m,k\in \N$ (with $(m,k)\not=(1,1)$ in the abelian case), and $\sigma$ a $v$-generic stability condition on $D^b(S)$. Then the moduli space $M_v$ (respectively $K_v$) is a symplectic variety with $\mathbb{Q}$-factorial singularities. If $(m,k)\neq (2,1)$, it has terminal singularities.
\end{prop}
\begin{proof}
When $S$ is a K3 surface, the first part of the proof is the content of \cite[Theorem~1.3, item~(a) and Corollary~6.9]{BM14projectivity}. Notice that the proof of the theorem by Bayer and Macr\`i actually shows that this moduli space is isomorphic to a moduli space of Gieseker-semistable twisted sheaves on a twisted K3 surface $(Y,\alpha)$, see \cite[Lemma 7.3]{BM14projectivity} for the existence of $(Y,\alpha)$ and the rest of that section for the isomorphism with a moduli space of twisted sheaves. The analogous result in the case of Abelian surfaces is the content of \cite[Theorem 1.4]{MYY13} (and notice that the assumption on the Picard rank in \cite[Theorem 1.4]{MYY13} is only for the case of K3 surfaces), where the authors prove that also in this case moduli spaces of stable objects for $v$-generic stability conditions are ordinary moduli spaces of stable twisted sheaves on a twisted abelian surface $(Y,\alpha)$ with Mukai vector $v'$.
The statement on the singularities in the abelian case now follows from \cite{PR18} and the vanishing of the Brauer class $\alpha$ for special surfaces: indeed we can take any family $(\mathcal{Y},\mathcal{\alpha})$ of twisted abelian surfaces containing $(Y,\alpha)$ and construct the relative moduli space of twisted sheaves with Mukai vector $v'$ (if the class of $v'$ remains algebraic). We are left with proving the terminality if $(m,k)\neq (2,1)$. To do so we need the following: 
\begin{claim}\label{claim:destab}
Let $\mathcal F^\bullet$ be a strictly $\sigma$-semistable object in $M_v$. Then all the elements in the Jordan-H\"older filtration of $\mathcal F^\bullet$ have Mukai vector $nw$ with $n<m$.
\end{claim}
\begin{proof}[Proof of Claim \ref{claim:destab}]
Suppose on the contrary that there exist an object $\mathcal F^\bullet$ and an element $\mathcal G^\bullet$ in its Jordan-H\"older filtration with Mukai vector $v'\neq nw$. By semistability, $v$ and $v'$ are not collinear, therefore by the proof of \cite[Proposition 9.3]{bridg}, there exists a wall $W_{v'}$ in the space of stability conditions to which $\sigma$ belongs. Therefore, $\sigma$ cannot be $v$-generic and the claim follows. This proof is completely analogous to the stable case in \cite[Proposition 9.3]{bridg}.
\end{proof}
Notice that the claim also proves that a general element of $M_v$ is stable.
Now, strictly semistable objects in $M_v$ are extensions of semistable objects with Mukai vector $nw$ and $(m-n)w$, with $1\leq n <m$. Hence, the codimension of the strictly semistable locus is
\[
\begin{aligned}
\mathrm{max}_{1\leq n<m} (m^22k+2-\mathrm{dim}(M_{nw})-\mathrm{dim}(M_{(m-n)w})&= \mathrm{max}_{1\leq n<m} (2k(-2n^2+2mn)-2)\\
&=2k(2m-2)-2\\
\end{aligned}
\]
  which is at least four if $(m,k)\neq (2,1)$, therefore by \cite[Corollary 1]{Nam01} it is terminal.
\end{proof}
Perego and Rapagnetta proved the following important result  if $v\neq (0,mH,0)$ or $S$ has Picard rank one for moduli spaces of sheaves. We  check here that for moduli spaces of objects one does not need to exclude this case.

\begin{thm}
[Theorem 1.19, \cite{PR18}]\label{thm:PR1}
Let $S$ be a projective $K3$ surface or an abelian surface, $v=mw$ a  Mukai vector with $w$ primitive with $w^2=2k$ for two non-zero integers $m,k\in \N$, $\sigma$ a $v$-generic stability condition on $D^b(S)$. 
Then $M_v$ (respectively $K_v$) is an irreducible symplectic variety and its locally trivial deformation type is determined by $(m,k)$. 
 \end{thm}
\begin{proof}
Notice that the walls in the stability manifold are bounded by \cite[Theorem 3.11]{maciocia}, hence by taking the large volume limit as in \cite[Section 14]{bridg}, we have a chamber ``at infinity'' where Bridgeland stability tends to Gieseker stability. If $v\neq (0,mH,0)$ or $S$ has Picard rank one, we remark that we can choose a stability condition so that the moduli space of stable objects is a moduli space of stable sheaves (by taking a stability condition in the final chamber under the large volume limit), hence this is precisely the content of \cite[Theorem~1.19]{PR18}. Assume now that $v=(0,mH,0)$ and $S$ does not have Picard rank one. To simplify the notation, we will restrict to the case of $M_v$, the argument however is literally the same for $K_v$. Inside the local deformation space $\mathrm{Def}(S)$, we consider the smooth hypersurface $B\subset \Def(S)$ of deformations of the pair $(S,H)$. This is a family of polarized K3 surfaces $f:\sS \to B$ and we denote by $\sH$ the polarization. By \cite[Theorem 24.1]{stab_fam}, we know that there is a relative stability condition $\sigma_B$ and we consider the moduli space $\gothM_\gothv:=M_{\gothv}(\sS,\sigma_B)\to B$ of $\sigma_B$-semistable objects on $\sS$ over $B$ with Mukai vector $\gothv=(0,m\sH,0)$ interpreted as a section of the relative Mukai lattice $R^{2*}f_*\Z_\sS$. One argues as in \cite[Lemma~2.15]{PR18} that $\gothM_\gothv  \to B$ is a deformation of $M_v$. If $(m,k)\neq (2,1)$, we have by Proposition~\ref{prop:terminal} that $M_v$ is terminal and $\mathbb{Q}$-factorial, hence all small deformations are locally trivial by \cite[Main Theorem]{Nam06} and locally trivial deformations preserve the property of being irreducible symplectic by Lemma~\ref{lemma irreducible symplectic locally trivial}. As the general fiber of $\sS\to B$ has Picard number one, we have a locally trivial deformation of $M_v$ to a variety where the theorem holds by the initial remark, and we are done. If on the other hand $(m,k)=(2,1)$, our varieties are $\mathbb{Q}$-factorial but not terminal, and they all have a resolution of singularities of relative Picard rank one by \cite[Proposition 2.2, Corollary 2.8]{MeachanZhang}, which is a blow up along their singular locus. Hence, $\gothM_\gothv  \to B$ is also a locally trivial deformation, as the $b_2$-constant locus coincides with the locally trivial locus by \cite[Proposition~5.13 and Corollary~5.14]{BL18}. 
We can now conclude again by Lemma~\ref{lemma irreducible symplectic locally trivial}. 
\end{proof} 
As a consequence of the proof, we have the following
\begin{cor}\label{cor:object_to_sheaves}
Let $S$ be a projective $K3$ surface or an abelian surface, $v=mw$ a  Mukai vector with $w$ primitive with $w^2=2k$,  for two non-zero integers $m,k\in \N$, $\sigma$ a $v$-generic stability condition on $D^b(S)$. 
Then $M_v$ (respectively $K_v$) is locally trivial deformation equivalent to the moduli space of sheaves $M_{(0,mH,0)}(S',H)$ (resp. $K_{(0,mH,0)}(S',H))$ where $S'$ is a K3 or abelian surface with $NS(S')=\mathbb{Z}H$ and $H^2=2k$.
\end{cor}

\begin{rmk}
In the  theorem above, we used moduli spaces of stable objects instead of moduli spaces of stable sheaves to avoid pathological situations for the Mukai vector $(0,mH,0)$ on a surface with Neron--Severi rank bigger than one. In our situation, a moduli space of stable objects for this choice of a Mukai vector is birational to the ordinary moduli space of sheaves, but the latter has possibly worse singularities. By \cite[Theorem 1.3]{BM14projectivity} and \cite[Theorem 1.4]{MYY13}, these moduli spaces of objects can be interpreted as moduli spaces of (twisted) sheaves on another surface, with a possibly different Mukai vector, hence the above theorem is exactly equivalent to the result stated in \cite{PR18}. 
\end{rmk}
\subsection{A construction due to Perego and Rapagnetta.}\label{ss:PR}
We recall below a construction introduced in \cite[Lemma~3.9]{PR18} which we will crucially use in the proof of Theorems~\ref{thm:amp} and~\ref{thm:exc}.

 Let $S$ be a projective $K3$ surface with an ample line bundle $H$ and consider the Mukai vector $v=(0, m\cdot H, 0),\ m\geq 1$.  
 Notice that a general element in the moduli space $M_v:=M_v(S,H)$ is an invertible sheaf of degree $g-1$ on a smooth curve $C \in |mH|$. In other words, $M_v$ contains as open dense subset the relative Picard variety 
\begin{equation}\label{eq relative picard}
\sJ^\circ:=\Pic^{g-1}(\sC^\circ/|mH|^\circ)
\end{equation}
where $\sC \to |mH|$ denotes the universal curve and the superscript denotes the restriction to the open subset $|mH|^\circ\subset |mH|$ parametrizing smooth curves. It comes with a Lagrangian fibration 
\begin{equation}\label{eq:lagr}
f: \sJ^\circ\to |mH|^\circ
\end{equation}
given by mapping a sheaf to its support.

Set $u:=(0, mH, 1-(g-1))$. Let $C\in |mH|$ be  an integral curve and $j : C\hookrightarrow S$ the inclusion. Then for every $L\in \Pic^1(C)$ the sheaf $j_*L$ is $H$-stable of Mukai vector $u$.  The sheaves of this type form an open subset U of $M_u$. If $L\in \Pic^1(C)$, then $L^{\otimes(g-1)}\in \Pic^{g-1}(C)$, hence $j_*L^{\otimes(g-1)}\in M_v$. 
The associated map 
\begin{equation}\label{eq:PR}
 \textrm{PR} : M_u\dashrightarrow M_v,\ j_* L\mapsto   j_*L^{\otimes(g-1)},
 \end{equation}
which is over the base $|mH|$,
is shown to be dominant, hence generically finite as both spaces have the same dimension, in \cite[Lemma 3.9]{PR18}. 

Observe that $u$ is a primitive Mukai vector. Hence, if every semistable sheaf is stable, e.g. if $S$ is general, then 
\begin{equation}\label{eq:smooth}
M_u \textrm{ is an irreducible holomorphic symplectic manifold of $K3^{[n]}$-type 
}
\end{equation}
by \cite{OGr97}, see also \cite[Theorem~6.2.5, Proposition~6.2.6]{HL10}. 

If $S$ is an abelian surface, then the construction is exactly the same, but replacing $M_v$ with $K_v$ and $M_u$ with $K_u$, see \cite[Theorem 0.2]{Yo-ab} and \cite[Theorem 4.4]{KLS06}.


\subsection{Moduli spaces of stable sheaves on K3 or abelian surfaces.}\label{ss:mod-sheaves}
If $(S,v,H)$ is an $(m,k)$-triple in the sense of \cite[Definition 1.15]{PR18}, then the moduli spaces of $H$-semistable sheaves $M_v$ of Mukai vector $v$ (respectively $K_v$ in the abelian case) are irreducible symplectic varieties by \cite[Theorem 1.19]{PR18}. In the remaining case $w=(0,w_1,0)$ and the Picard number $\rho(S)>1$, we use  the construction by Perego and Rapagnetta recalled above to show the following result, which will be only marginally used in the paper (cf. Remark \ref{rmk:one}), but may be interesting in its own right. 
\begin{prop}\label{prop:not-triple}
 Let $S$ be a projective $K3$ surface or an abelian surface, $v$ a  Mukai vector, $H$ an ample divisor on $S$, and $m,k\in \N$ two positive integers. Suppose that
 \begin{enumerate}
\item[(i)] the polarization $H$ is primitive and $v$-generic;
\item[(ii)] $v=mw,$ with $w$ primitive and $w^2=2k$.
\end{enumerate}
Then the moduli space of stable sheaves $M_v$ (respectively $K_v$ in the abelian case) is a primitive symplectic variety. 
 \end{prop}
 \begin{proof} 
 We may assume $S$ is a projective $K3$ surface with $\rho(S)>1$ and consider the Mukai vector $v=(0, m\cdot H, 0),\ m\geq 1$.  
We set $u:=(0, mH, 1-(g-1))$ and consider
the dominant Perego--Rapagnetta map recalled above
\begin{equation}
 \textrm{PR} : M_u\dashrightarrow M_v,\quad j_* L\mapsto   j_*L^{\otimes(g-1)}. 
 \end{equation}
Even if $M_u$ is not an IHS manifold
we do nevertheless have the following. 

\begin{claim}\label{claim:stab}
There exists a stability condition $\sigma$ which is $u$-generic such that the moduli space of $\sigma$-stable objects $M_\sigma$ provides a crepant resolution $\nu : M_\sigma \to M_u$ and $M_\sigma$ is of K3\textsuperscript{[n]}-type.
\end{claim}
\begin{proof}[Proof of Claim \ref{claim:stab}]
Let us consider the component $\mathrm{Stab^\dagger}(S)$ of stability conditions in $D^b(S)$ defined by Bridgeland in \cite[Theorem 1.1]{bridg}, which contains the Gieseker stability condition given on $M_u$ by $H$. Notice that the walls in the stability manifold are bounded by \cite[Theorem 3.11]{maciocia}, hence by taking the large volume limit as in \cite[Section 14]{bridg}, we have a chamber ``at infinity'' where Bridgeland stability tends to Gieseker stability. Let $U$ be an open set in the closure of $\mathrm{Stab^\dagger}(S)$ containing the $H$-Gieseker stability condition such that the closure $B$ of $U$ is compact. Let us consider all semistable objects for some $\sigma\in U$ in $D^b(S)$ with Mukai vector $u$. This set of objects has bounded mass in the sense of \cite[Definition 9.1]{bridg}, as the masses of a set of objects depends on the Mukai vectors of their subobjects, which in our case form a finite set. Therefore by \cite[Proposition 9.3]{bridg}, there is an open set of stability conditions in $B$ such that all semistable objects are stable. Let $\overline{\sigma}$ be one such condition, such that it lies in a chamber of $B$ whose closure contains the $H$-Gieseker stability condition. Notice that such a chamber exists by taking the so called large volume limit (see \cite[Theorem 6.7]{BM14projectivity}).  Therefore, the space of stable objects $M_u(S,{\overline{\sigma}})$ is smooth and projective by \cite[Corollary 6.9 and 7.5]{BM14projectivity}. As semi-stability is preserved in the closure of the chamber containing $\overline{\sigma}$ by continuity of the mass function (see \cite[before Lemma 2.2]{bridg}), we have a morphism $M_u(S,{\overline{\sigma}})\to M_u(S,H)$ which contracts $S$-equivalence classes of strictly $H$-semistable objects. Finally, $M_u(S,{\overline{\sigma}})$ is of $K3$\textsuperscript{[n]}-type by \cite[Section 7]{BM14projectivity}, as it is birational to a moduli space of Gieseker stable twisted sheaves on a K3 surface (derived equivalent to $S$).
\end{proof}

Hence, either directly by \eqref{eq:smooth} or by Claim \ref{claim:stab} we have a dominant rational map from an irreducible holomorphic symplectic manifold $Y$ to $M_v$. We already know by \cite{Muk84} that $M_v$ carries a reflexive $2$-form which is symplectic on the regular part. If $M'_v\to M_v$ is a resolution of singularities, we deduce furthermore that $H^0(M_v, \Omega^{[2]}_{M_v})$ is generated by the symplectic form by the following 
\begin{equation}\label{eq inequality twoforms}
h^0(\Omega^{[2]}_{M_v})=h^0(\Omega^2_{M'_v})\leq h^0(\Omega^2_Y)=1
\end{equation}
where the first equality follows from \cite{GKKP11} and the inequality from the fact that we have a dominant rational map from $Y$ to ${M'_v}$. It also follows that $M_v$ has symplectic singularities in the sense of Definition~\ref{def:symp-sing} by applying Lemma~\ref{lemma galois closure} below to a resolution $Z\to M_v'$ of indeterminacy of $Y\ratl M_v$. Note that $Y$ is smooth and symplectic. Moreover, if $M'_v\to M_v$ is a resolution of singularities, we have 
\begin{equation}
h^1(\cO_{M_v})=h^1(\cO_{M'_v})=h^0(\Omega^1_{M'_v})\leq h^0(\Omega^1_Y)=0
\end{equation}
and the conclusion follows.

In the abelian case the proof is identical, once one uses \cite{MYY13} instead of \cite{BM14projectivity}. 
 \end{proof} 

\begin{lemma}\label{lemma galois closure}
Let $h:Z \to X$ be a surjective generically finite morphism from a smooth projective variety to a normal projective variety. If $\sigma$ is a $k$-form on $X_\reg$ whose pullback along $h$ extends to a regular $k$-form on Z, then for any resolution $\pi:X'\to X$ the pullback $\pi^*\sigma$ extends to a regular $k$-form as well.
\end{lemma}
\begin{proof}
Replacing $Z$ by a further birational modification, we may assume that $f$ lifts to $X'$. Now we consider a diagram
\[
\xymatrix{
\wt Z \ar@/^4mm/[rr]\ar[d] & Z \ar[r]\ar[d] &  X' \ar[d]^\pi \\
\wt Y \ar[r] & Y \ar[r]^f & X\\
}
\]
where $Z\to Y$ is the Stein factorization of $h$, the morphism $\wt Y \to Y$ is the Galois closure of $f$, and $\wt Z \to \wt Y$ is an equivariant resolution of singularities for the action of the Galois group $G$ on $\wt Y$ that also admits a generically finite morphism $\wt Z \to X$. Such a resolution be obtained as follows. Composition with $\pi^{-1}$ gives a rational map $\wt Y \ratl X'$. Then the closure of the graph of this map is a closed subvariety $\Gamma \subset \wt Y \times X'$ which has a $G$-action (trivial on the second factor, note that $\Gamma$ is $G$-stable as $\wt Y \to X$ is $G$-invariant) and a morphism to $X$. Then $\wt Z$ can be obtained from $\Gamma$ by taking an equivariant resolution. Now we claim:
\begin{enumerate}
	\item The pullback of $\sigma$ to $\wt Z$ extends as a regular $k$-form $\wt \sigma$ on $\wt Z$.
	\item This form descends to $X'$.
\end{enumerate}
For the first claim, it suffices to note one can always pull back forms along rational maps between smooth varieties. As $f^*\sigma$ extends to $Z$ by assumption, the claim follows. For the second claim we observe that $\wt Z \to X'$ is $G$-invariant and factors thus through the quotient by $G$. Being a pullback from $Y$, the form $\wt\sigma$ is invariant and thus descends back to $\wt Z/G$ (as a reflexive $k$-form). Finally, any reflexive $k$-form descends along the birational contraction $\wt Z/G \to X'$.
\end{proof}

%
\section{Deforming rational curves on primitive symplectic varieties}\label{section rational curves np}
%
The goal of the section is to extend some  results in the smooth case concerning the deformation theory of rational curves to primitive symplectic varieties.

The general framework will be the following. Given a compact K\"ahler variety $X$ we will denote by $\overline{{ M_0}}(X, \alpha_0)$ the Kontsevich moduli stack of genus zero stable maps into $X$ of class $\alpha_0\in H_2(X,\Z)$. If $f$ is a stable map, we denote by $[f]$ the corresponding point of the Kontsevich moduli stack. 
We refer to \cite{BehrendManin96, FultonPandharipande95, AbramovichVistoli02} for details and constructions in the projective case and to e.g. Jason Starr answer's in \cite{Overflow} for a discussion of the Kontsevich moduli stack of genus zero stable maps in the K\"ahler case. We will often consider the relative situation as follows. 
Let $\pi : \mathcal X\ra B$ be a proper morphism of complex varieties whose fibers are compact K\"ahler varieties of dimension $2n$, and let $\alpha$ be a global section of the local system $R^{4n-2}\pi_*\Z$. Suppose that $\alpha$ is fiberwise of Hodge type $(2n-1,2n-1)$. Consider the relative Kontsevich moduli stack $\overline{{\mathcal M_0}}(\mathcal X/B, \alpha)$ which parametrizes genus zero stable maps $f : C\ra X$  to fibers $X=\mathcal X_b$, $b\in B$, of $\pi$ such that $f_*[C]=\alpha_b$. 
The canonical morphism $\overline{{\mathcal M_0}}(\mathcal X/B, \alpha)\ra B$ is proper. 
\begin{definition}\label{definition uniruled}
We say that an irreducible subvariety $Z\subset X$ is {\it uniruled} if there exists an irreducible subvariety $T\subset \overline{{ M_0}}(X, \alpha_0)$ such that the evaluation morphism 
$$
 {\textrm ev}_T: \mathcal C_T\to Z \subset X
$$
restricted to the universal curve $\mathcal C_T\to T$ over $T$ is dominant. We refer to such a component $T$ as a \emph{ruling}. By {\it a general curve in the ruling of $Z$} we mean the morphism
$$
 {\textrm ev}_t : C_t\to Z
$$
for general $t\in T$.
\end{definition}

\begin{rmk}\label{rmk:uniruled-np}
Notice that a ruling does not have to be unique. Nevertheless, using the MRC-fibration and the symplectic form, it is easy to see that there is a unique covering family of \emph{irreducible} rational curves if the divisor is irreducible. Let us assume for simplicity that $D$ is projective. Then from \cite[Part I, Theorem~V.3.1]{MP97} we infer that there exists a proper modification $\nu:\wt D \to D$ from a smooth variety $\wt D$ and a proper morphism $p:\wt D \to B$ with rationally connected fibers whose very general fiber contains all rational curves it meets. The map $p$ is a  resolution of indeterminacy of the MRC-fibration; its existence in the nonprojective case is a direct consequence of \cite[Th\'eor\`eme~1.1]{Ca81}. As $\wt D$ is smooth, the general fiber of $p$ is irreducible. Since the pull back $\nu^*\sigma$ of the symplectic form has generically a one dimensional radical, $p$ has relative dimension one and the images of fibers of $p$ constitute the sought-for unique covering family. The expression {\it the general curve of the ruling of a uniruled divisor $D$} will then refer to the general curve in this unique covering  family of {irreducible} rational curves. 

Clearly, there can still be different rulings in the sense of Definition~\ref{definition uniruled}. However, if there is a birational contraction whose exceptional locus is the given irreducible uniruled divisor, there is a distinguished ruling determined by the family of general fibers of the contraction. Notice that by what we said before, the contraction is necessarily unique. Such contraction exists on a birational model for prime exceptional divisors, thanks to Theorem \ref{thm:exc-def}, item (1).
\end{rmk} 

The argument presented in  Remark \ref{rmk:uniruled-np} immediately yields the following. 
\begin{lem}\label{prop:voisin}
Let $X$  be primitive symplectic variety of dimension $2n$. Let 
$D\subset X$ be a uniruled divisor. Then through the general point of $D$ there are finitely many rational curves. 
\end{lem}

If $X$ is {\it smooth}, it is well known (see e.g. \cite{BHT11}) that 
\begin{equation}\label{eq:dim}
\dim_{[f]} \overline{{ M_0}}(X, f_*[C]) \geq \dim(X) +\deg f^*T_X - 3.
\end{equation}

We place ourselves in the following

\begin{setting}\label{setting rational curves} 
We assume that $X$ is a primitive symplectic variety of dimension $2n$, $C$ is a connected union of normally crossing smooth rational curves, and $f : C \to X$ is a morphism which is generically injective on each irreducible component of $C$. We suppose furthermore that no intersection between two irreducible components of $C$ is sent to the singular locus of $X$. If needed, we will denote by $R$ the image of $f$ in $X$. 
\end{setting}

It is well known that if $X$ is a smooth IHS of dimension $2n$, the lower bound (\ref{eq:dim}) can be improved by one, namely $\overline{{ M_0}}(X, f_*[C])$ has dimension at least $2n-2$ and moreover $R$ deforms along its Hodge locus in $\Def(X)$ whenever the equality holds (cf. \cite{Ran95,CMP19, AV15}).

Let ${\mathcal X}\rightarrow \Def^\lt(X)$ be a local universal family of locally trivial deformations of $X$.
Recall that  $\Def^\lt(X)$ is smooth by \cite[Theorem 4.11]{BL18}. 
 Let $B$ be the Hodge locus associated to $[R]$ in $\Def^\lt(X)$. In particular, by the local Torelli theorem \cite[Corollary 5.8]{BL18}, $B$ is a smooth divisor in $\Def^\lt(X)$. By abuse of notation we will denote by $[R]$ the corresponding  global section of Hodge type $(2n-1,2n-1)$ of the local system $R^{4n-2}\pi_*\Z$.
 
Recall from \cite{GKK10}, that there always exist strong log resolutions $\pi:Y\to X$ such that $\pi_*T_Y=T_X$ for any reduced complex space $Y$.

\begin{prop}\label{prop:defcurves}
Suppose we are in Setting~\ref{setting rational curves}, consider a local universal family ${\mathcal X}\rightarrow S\subset \Def^\lt(X)$ of locally trivial deformations of $X$, and let $\Pi:\sY \to \sX$ be a simultaneous resolution of singularities that is an isomorphism on the smooth locus of $\sX\to S$. We denote by $\pi:Y\to X$ the central fiber of $\Pi$. Suppose that $R$ is not contained in the singular locus of $X$. Let $g:C\to Y$ a lift of $f$ and $M$ be an irreducible component of $\overline{{ M_0}}(\mathcal X, f_*[C])$ containing $[f]$. Then the following holds:
\begin{enumerate}
\item\label{item dimension prop:defcurves} $M$ has dimension at least $2n-2- \deg (g^*K_Y)$;
\item\label{item noether lefschetz prop:defcurves} If $M$ has dimension $2n-2- \deg (g^*K_Y)$, then (possibly after shrinking the representative of $\Def^\lt(X)$) any irreducible component of $\overline{{\mathcal M_0}}(\mathcal X/B, [R])$
 containing $M$ dominates the Hodge locus $B$ where the class $[R]$ remains algebraic.
\end{enumerate}
\end{prop}

Note that such a lift $g$ exists and is unique as we are in Setting \ref{setting rational curves}. Moreover, the term $\deg (g^*K_Y)$ is always non-negative as $X$ has canonical singularities, and a simultaneous resolution as in the statement of the proposition always exists by \cite[Corollary~2.23]{BGL20}. In fact, it can be obtained as a deformation of any given resolution $\pi:Y\to X$ satisfying $\pi_*T_Y=T_X$.

\begin{proof}
To prove (1), observe that a generic deformation of $f$ can be lifted to $Y$, i.e. there exists an irreducible component of $\overline{{ M_0}}( Y, g_*[C])$ containing $[g]$ together with a birational map $M'\ratl M$.  Let $T\subset S$ be a smooth curve passing through the point $0\in S$ corresponding to $X$ and not contained in the Hodge locus $B:=\Hdg_{[R]}(X) \subset S$ of $[R]$. Recall from \cite[Lemma~4.13]{BL18} that $B$ is a divisor in $S$. Consider the restrictions to $T$
\begin{equation*}
\xymatrix{
\mathcal Y_T\ar[r]\ar[dr]& \mathcal X_T\ar[d]\\
& T
}
\end{equation*}
 of $\Pi$. 

Let $\mathcal M_T\to T$ be an irreducible component of $\overline{{\mathcal M_0}}(\mathcal X_T, [R])$ that contains $M$. Lifting rational curves from $\mathcal X_T \to T$ to $\mathcal Y_T \to T$, we determine an irreducible component $\mathcal M'_T\to T$ of $\overline{{\mathcal M_0}}(\mathcal X_T, [R])$ which dominates $\mathcal M_T$ generically finitely.

By (\ref{eq:dim}) we have 
\begin{eqnarray*}
\dim_{[f]}  \mathcal M_T = \dim_{[g]}  \mathcal M'_T &\geq& -\deg g^*K_{\mathcal Y_T} + \dim {\mathcal Y_T} - 3\\
&=& -\deg g^*K_Y + 2n - 2.
\end{eqnarray*}
Now by the choice of $T$ outside the 
Hodge locus of $[R]$, for $0\neq t\in T$ close to $0$, no deformation of $f$ parametrized by $\mathcal M_T$ can be contained in $X_t$ and item (1) follows.

To show (2), consider as above an irreducible component $\mathcal M\to S$ of $\overline{{\mathcal M_0}}(\mathcal X, f_*[C])$ containing $M$ and  again let $\mathcal M'\to S$ be the unique irreducible component of $\overline{{\mathcal M_0}}(\mathcal Y, g_*[C])$ obtained by lifting rational curves to the desingularizations. 

By (\ref{eq:dim}) we have 
\begin{equation*}
\dim_{[g]} \mathcal M' \geq  -\deg g^*K_{\mathcal Y} + \dim {\mathcal Y} - 3\geq \dim S + 2n - 3 - \deg (g^*K_Y),
\end{equation*}
from which we deduce 
$$
\dim_{[f]} \mathcal M\geq \dim S + 2n - 3 - \deg (g^*K_Y)= \dim B + 2n - 2 - \deg (g^*K_Y).
$$
Since the image in $S$ of $\mathcal M$ is necessarily contained in the divisor $B$ given by the Hodge locus  and the fibers of such a component all have dimension at least $2n-2- \deg (g^*K_Y)$ by item (1), if a fiber has dimension $2n - 2- \deg (g^*K_Y)$, it follows from semi-continuity of the fiber dimension that $\mathcal M$ has to dominate B, which shows the result.
\end{proof}

In practice, Proposition \ref{prop:defcurves} can be useful to deform families of rational curves covering a divisor only when the general member of the family does not intersect the singular locus as the following elementary dimension count shows. 

\begin{lem}\label{lem:lower}
Let $X$  be a  primitive symplectic variety of dimension $2n$. Let $f:C\to X$  a genus zero stable map whose deformations in $X$ cover an irreducible divisor $D$. Then $f$ deforms in a family of  dimension at least $2n-2$. 
\end{lem}
\begin{proof}
Let $M$ be an irreducible component of $\overline{{ M_0}}( X, f_*[C])$ containing $[f]$ and let $\mathcal C\to M$ be the (domain of the) universal family restricted to $M$.  If by contradiction $\dim(M)=2n-2-e$ for some $e>0$, then $\mathcal C$ would have dimension $2n-1-e$ and  the evaluation morphism $\ev :\mathcal C\to D$ cannot be dominant.
\end{proof}

\begin{prop}\label{prop:expdim}
Let $X$  be a  primitive symplectic variety of dimension $2n$. Let $D\subset X$  be an irreducible uniruled divisor and $f:C\to D$  a general curve in the ruling (cf. Remark \ref{rmk:uniruled-np}). 
Then $f$ deforms in a family of  dimension $2n-2$.
\end{prop}
\begin{proof}
Let $f:C\to D\subset X$ be a curve in the family passing through the general point of $D$. 
We argue by contradiction. By Lemma \ref{lem:lower} we may then suppose that there exists an irreducible component  $M$ of $\overline{{ M_0}}( X, f_*[C])$ containing $[f]$  and having dimension $\geq 2n-1$. 
Consider the universal curve $\mathcal C\to M$ and the evaluation morphism
$$
ev: \mathcal C\to D\subset X.
$$
Since $\dim \mathcal C\geq 2n$, through the general point of $D$  we would have a family $\mathcal C_x\to M_x$ of rational curves with $\dim (M_x)\geq 1$. This would contradict Lemma \ref{prop:voisin} and we are done.
\end{proof}

The arguments in Remark~\ref{rmk:uniruled-np} in particular implied that only finitely many curves pass through a general point in a uniruled divisor. The following strengthening of that statement will be crucial in the proof of Theorem~\ref{thm:uniruleddivisor} and also yields Corollary~\ref{corollary term} below.

\begin{thm}\label{theorem:term}
Let $X$ be a compact K\"ahler variety with rational singularities, let $D\subset X$  be an irreducible uniruled divisor, and let $\sigma$ be a reflexive $2$-form on $X$ which is symplectic at the general point of $D$. If $\Sigma \subset X$ is a closed subvariety such that every curve in the ruling meets $\Sigma$, then $\codim_X\Sigma \leq 2$. 
\end{thm}
\begin{proof}
Without loss of generality we may assume that $\Sigma$ is a proper subset of $D$ and is irreducible. We fix the ruling given by Remark~\ref{rmk:uniruled-np} and assume that 
\begin{equation}\label{eq:every}
\textrm {every curve in the ruling meets } \Sigma. 
\end{equation}
From this, we will deduce that the rank of the pull back of the holomorphic $2$-form $\sigma$ on $X$ along the inclusion $\iota:D_\reg \cap X_\reg \to X_\reg$ is $< 2n-2$ at a general point where $\dim X=2n$. This is impossible, see \cite[Lemma 3.3]{CMP19}, and will conclude the proof of the theorem by contradiction.

The rank of $\iota^*\sigma$ can be determined after pull back along any surjective morphism $\cC' \to D$ as long as $\cC'$ is reduced (since such a morphism will be smooth at the generic point of $\cC'$). We will construct a generically finite such map together with a fibration $f:\cC'\to \Sigma$ with fiber dimension $\geq 2$ such that the pullback of $\sigma$ to $\cC'$ is generically a pullback from $\Sigma$. Let
\[
\xymatrix@C=12mm@R=8mm{
\cC \ar[r]^(0.4){e}  \ar[d]   &  D\supset \Sigma \\
T& \\
}
\]  
be the ruling family of curves on $D$ where the parameter variety $T$ and the family $\cC$ of curves over it are compact and irreducible, in particular, the evaluation morphism $e$ is surjective.  Here, we choose $\cC \to T$ to be smooth over a dense open set in $T$, i.e. $e$ is the normalization when restricted to the generic fiber of $\cC \to T$. 

Notice that, although by assumption every curve $C_t,\ t\in T,$ intersects $\Sigma$, we do not in general obtain a map $T\to \Sigma$ sending $t$ to a point of intersection of the corresponding curve with $\Sigma$ as there could be many points of intersection. To remedy this we proceed as follows. 
 Let $T'$ be a resolution of singularities of an irreducible component of $e^{-1}(\Sigma) \subset \cC$ such that the induced maps $T' \to \Sigma$ and $T'\to T$ are both dominant (note that $e^{-1}(\Sigma)$ could have several components). Then $T'\to T$ is finite (as the generic ruling curve will not be contained in $\Sigma$) and $T'\to \Sigma$ has relative dimension equal to $\codim_D \Sigma -1$. 
This last claim follows because all curves in the ruling meet $\Sigma$ by assumption. We denote by $\cC'$ the unique irreducible component of $T'\times_T \cC$ that dominates $\cC$ and have thus obtained the following diagram.
\begin{equation}\label{eq:key-diag}
\xymatrix{
\cC' \ar@/_2em/[dd]_f \ar[d]_a\ar[r]^p & \cC \ar[d]\ar@/^1.5em/@[blue][ddr]^{e}  &&\\
T' \ar[d]_b\ar[r] \ar@/_/[u]_s& T & &\\
\Sigma \ar@{^(->}@[red][rr]& & D \ar@{^(->}[r]^\iota & X\\
}
\end{equation}
Here, $s$ is the canonical section of $a$. Note that the diagram is not commutative. This is because $\cC'$ dominates $D$ via $e':=e\circ p$, but $f$ maps to $\Sigma$ which is strictly contained in $D$. However, the diagram becomes commutative if we either delete the red or the blue arrow. We replace $\cC'$ by a resolution of singularities and assume henceforth that it is smooth. If we denote by $\sigma_{\cC'}$ the pull back (of reflexive forms, see \cite[Theorem~14.1]{KS18} -- here we use the hypothesis of rationality of the singularities of $X$) of $\sigma$ along $\iota\circ e'$, it remains to show that $\sigma_{\cC'}$ is a pull back along $f$. By construction, $a$ is a smooth $\P^1$-fibration over a Zariski open set $U\subset T'$, so clearly the restriction $\sigma_{\cC'}\vert_{a^{-1}(U)}$ descends to a holomorphic $2$-form on $U$. This coincides however with $s^*\sigma_{\cC'}\vert_U$ and therefore has a holomorphic extension to the whole of $T'$. As $\cC'$ is smooth and $\sigma_{\cC'}-a^*s^*\sigma_{\cC'}$ is torsion, we obtain that $\sigma_{\cC'} = a^*\eta$ where $\eta=s^*\sigma_{\cC'}$ is a holomorphic $2$-form on~$T'$. 

Now consider the commutative diagram obtained from (\ref{eq:key-diag}) by deleting the blue arrow. We have to show that $\eta$ descends along $b$. But this is now immediate: by construction, $\eta$ is the reflexive pullback along $T'\to X$ (recall that $T'$ was chosen nonsingular) and this morphism factors through $T'\to \Sigma$.
\end{proof}

\begin{corollary}\label{corollary term}
Let $X$  be a  primitive symplectic variety {\it with terminal singularities}. 
Let $D\subset X$  be an irreducible uniruled divisor. Then the general curve $R$ in the ruling of $D$ (cf. Remark \ref{rmk:uniruled-np}) does not meet $X_{\sing}$. 
\end{corollary}
\begin{proof}
This follows directly from Theorem~\ref{theorem:term} and the fact that by \cite[Corollary 1]{Nam01}, $X$ has terminal singularities if and only if $\codim ( X_{\sing})\geq 4$. 
\end{proof}

\begin{remark}\label{remark projective case}
It is worth noting that the proof of the above theorem is much easier in the projective case. Let us take a resolution of singularities $\pi:Y\to X$. By a result of Hacon--M\textsuperscript{c}Kernan, the fiber $\pi^{-1}(x)$ over every $x\in X$ is rationally chain connected by \cite[Corollary~1.5]{HM07}. If $\{C_t\}_{t\in S}$ is a family of rational curves all passing through a given point $x\in D$ and such that no curve is contained in $X_\sing$, we see that all points in the strict transforms $\wt C_t$ lie in the same ``rational orbit'', i.e. are equivalent under rational equivalence. Then the result follows easily from Mumford's theorem, see e.g \cite[Proposition 22.24]{Voisin2002}. Note that instead the proof of Theorem~\ref{theorem:term} does not need the result of Hacon--M\textsuperscript{c}Kernan, but uses Kebekus--Schnell \cite[Theorem 14.1]{KS18} instead. Note that the projective precursor \cite{Keb13} of the latter used \cite{HM07} in the proof.
\end{remark}

\begin{remark}\label{remark rcc nonprojective case}
Even though it is not needed in this article, it is interesting to note that there is an analog of \cite[Corollary~1.5]{HM07} for non-projective primitive symplectic varieties. First note that it is sufficient to show this for one resolution of singularities because the statement holds for smooth varieties and given two resolutions we can always find a third one dominating both.

Thus, we may choose a resolution $\pi : Y\to X$ satisfying $\pi_*T_Y=T_X$. Hence, we obtain a simultaneous resolution of singularities $\cY\to \cX=\{\pi_t:Y_t \to X_t\}_{t\in \Delta}$ deforming $\pi$ by \cite[Corollary~2.23]{BGL20} so that in $\Delta$ we have a dense set $\{t_i\}_{i\in I}$ of projective points (use \cite[Corollary~6.11]{BL18}). Notice that the simultaneous resolution is in fact locally trivial as a morphism (i.e. has local trivializations over $X$). In particular, the fibers of the resolution deform globally trivially. Therefore, the fiber $\pi^{-1}(x)$ is rationally chain connected as the fibers $\pi_{t_i}^{-1}(x)$ are. 
\end{remark}

We can now prove our main deformation-theoretic result. 

\begin{proof}[Proof of Theorem \ref{thm:uniruleddivisor}]
(1) By Proposition~\ref{prop:expdim}, the map $f$ deforms in a family of dimension $2n-2$. By Theorem~\ref{theorem:term}, the general deformation of $f$ in $X$ can meet $X_\sing$ only in its codimension two irreducible components, and from Proposition~\ref{prop:kaledin} we infer that the singularities of $X$ are transversal ADE-surface singularities along such components. As observed in \cite[Lemma~4.9]{BL18}, the functorial resolution of Bierstone--Milman and Villamyor \cite{BM97,Vil89} applied to the locally trivial family $\sX\to S$ gives a simultaneous resolution $\Pi:\sY \to \sX$ (here we need that $S$ is reduced). We also refer to \cite[Chapter~3]{Kol07} and \cite[Section~4]{GKK10} for properties of this resolution. By functoriality with respect to smooth morphisms we conclude that $\Pi$ is crepant over the locus of transversal ADE singularities. Now we can lift the generic deformation of $f$ in $X$ along $\Pi$ and apply Proposition~\ref{prop:defcurves}. Note that the term $\deg g^*K_Y$ vanishes because $\Pi$ is crepant at the generic deformation of $[f]$. 
Then $\mathcal M$ is smooth at a general deformation of $[f]$ because the dimension equals the expected dimension. 

To show statement (2), one can argue as in the smooth case. 
It suffices to consider the case where $B$ has dimension $1$ and passes through a very general point of the Hodge locus.
Let $\mathcal M$ be an irreducible component  of  the space of relative genus zero stable maps $\overline{\mathcal{ M}_0}(\mathcal X/B, f_*[C])$ in the local universal family $\mathcal X\to B\subset \Def^\lt(X)$ of locally trivial deformations $\mathcal{X}_t$ of $X$. Suppose $\mathcal M$ contains $M$,
Then $\mathcal M$ dominates $B$ by the previous item. Let $\mathcal D\subset \mathcal X\to B$ be the locus in $\mathcal X$ covered by the deformations of $f$ parametrized by $\mathcal M$. Since $\mathcal M$ dominates $B$, any irreducible component of $\mathcal D$ dominates $B$. Since the fiber of $\mathcal D\to B$ over $b_0$ is a divisor in $\mathcal X_{b_0}=X$, this means that at a smooth point the evaluation morphism
has maximal rank. Therefore, it remains maximal over an open set and as a consequence
 the fiber of $\mathcal D\to B$ at a general point $b$ is a divisor, which is uniruled by construction. To see that for every $b$ the corresponding variety $\mathcal X_b$ contains a uniruled divisor it suffices 
 to argue with the irreducible component of the relative Hilbert scheme which contains
 $D$. By the above such a component surjects onto the base $B$ and yields a uniruled divisor in each fiber.
 
 The N\'eron-Severi group of a general fiber $\mathcal X_b$ is generated by the dual of $[R]$, which shows the result on the cohomology class. Notice however that we cannot deduce that the class of the divisor $D_b$ inside $\mathcal X_b$ equals that of $D$. Indeed, the restriction $\mathcal M_{b_0}$ of $\mathcal M$ over $b_0$ could well be reducible 
$$
\mathcal M_{b_0}= M \cup (\bigcup_j  M_j),
$$ and therefore $\mathcal D_{b_0}$ could contain $D$ as well as all the other images $D_j$ of the evaluation morphisms 
associated to the other components $M_j$.
Nevertheless, the curves in the ruling of each of the  $D_j$'s lie in the same homology class $\alpha$.
Therefore, by deforming to a general point $b$ where the N\'eron-Severi group is generated by the dual of $\alpha$, 
we see that the divisor $\mathcal D_{b}$ has class proportional to a (positive) multiple of the dual of $\alpha$ and we are done. 
\end{proof}

\begin{remark}\label{remark terminal simplification}
Let us point out that if $X$ is terminal, the conclusion of item (1) of Theorem~\ref{thm:uniruleddivisor} follows directly from Corollary~\ref{corollary term} and  Proposition~\ref{prop:defcurves}, item~(2). In the general situation, the existence of a $\Q$-factorial terminalization would simplify the proof slightly to the effect that one would not need any properties of the functorial resolution but instead work with a simultaneous $\Q$-factorial terminalization (obtained by deforming a $\Q$-factorial terminalization of $X$, see \cite[Remark~2.29]{BGL20}) and apply Corollary~\ref{corollary term} to it. For primitive symplectic varieties, one possibly needs to pass to a different birational model in order to have a $\Q$-factorial terminalization, see \cite[Theorem~9.1]{BL18}, and therefore it seemed easier to work over the given variety at the cost of possibly introducing discrepancies.
\end{remark}

In the case of a prime exceptional divisor, i.e. an irreducible and reduced divisor $E$ on $X$ such that $q_X(E)<0$, using the Minimal Model Program we can do slightly better and deform at least a multiple of the initial divisor.   We refer the reader to e.g. \cite{BCHM10} and \cite{BBP13} for the relevant definitions that we will need in the proof. 

\begin{proof}[Proof of Theorem \ref{thm:exc-def}]
(1) Since $X$ has canonical singularities we can choose a rational $0<\epsilon \ll 1$ such that the pair $(X, \epsilon E)$ is klt. 
By e.g. the Boucksom--Zariski decomposition \cite[Theorem 1.1]{KMPP19} applied to $\epsilon E$ we see that it coincides with its negative part, and is also equal to the restricted base locus $\mathbb B_-(E)$. Recall that
\[
\mathbb B_-(E):=\bigcup_{A \textrm{ ample } \Q-\textrm{divisor}} \B(E+A)
\]
where $\B(\cdot)$ stands for the stable base locus. Now we can argue exactly as in the proof of \cite[Theorem A]{BBP13}, see also \cite[Theorem 3.3]{Dru11}. We recall the proof to record further information that will be used below.
By \cite[Lemma 1.14]{ELMNP06} there exists a  $\Q$-Cartier ample effective  divisor $A$ on $X$ such that $E$ is a component of the augmented base locus 
$\mathbb B_+(\epsilon E+A)$. Since $A$ is ample, we may furthermore assume that $(X, \epsilon E+A)$ is still klt. Using \cite[Corollary 1.4.2]{BCHM10} it is shown in \cite[Proof of Theorem A]{BBP13} that the MMP directed by $A$ leads to a birational model $f:X\dashrightarrow X'$ of $X$, together with a birational proper morphism $c:X'\to W'$ of relative Picard number $\rho(X'/W')=1$ whose exceptional locus coincides with the strict transform $E'$ of $E$.  
Therefore, $E'$ is uniruled by \cite{Kawa91} and since $\rho(X'/W')=1$ the dual to $E'$ is proportional to the class of the general curve $R'$ in its ruling. Therefore, the same conclusions hold for $E$, namely $E$ is uniruled and, since $f_*$ is an isometry, its dual $E^\vee$ is proportional to the class of a general curve $R$ of its ruling. The variety $X'$ is a locally trivial deformation of $X$ by the $\Q$-factoriality hypothesis and \cite[Theorem 6.17]{BL18}. 

Proposition \ref{prop:kaledin} implies that either $R'$ is a smooth rational curve or a union of two smooth rational curves meeting transversally. Indeed, the singularities of $W'$ are generically transversal ADE surface singularities along the image of the exceptional locus of $c:X'\to W'$. Hence, the dual intersection complex of a general fiber is a subgraph of an ADE graph. As the exceptional divisor is irreducible and the monodromy acts by a graph automorphism, the possibilities for $R'$ are as claimed.

Since, as observed above, $R'$ does not meet the indeterminacy locus of $f^{-1}$, the same holds for $R$. Moreover, for each irreducible component $C'$ of $R'$ we have
\begin{equation}\label{eq:norm}
N_{C'/X'} = \omega_{C'}\oplus \mathcal O_{C'}^{\oplus (2n-2)}
\end{equation}
The claim about the primitivity of $E$ now follows from the fact that 
$$
 E\cdot R = \deg (\omega_R)=-2.
$$

(2.a) Let $R$ be a general curve in the ruling of $E$ given by item (1). 
Let $\sX \to B:=\Hdg_{[R]}(X) \subset \Def^\lt(X)$ be the restriction of the Kuranishi family of locally trivial deformations of $X$ to the Hodge locus of $R$. We know by Theorem~\ref{thm:uniruleddivisor} that the deformations of $R$ inside $\mathcal X$ continue to cover a divisor $E_t$ on $\sX_t$. This divisor specializes to a divisor $E_0 \subset \sX_0=X$ such that $E\subset E_0$.  Let us write $E_0=mE+F$ for some $m>0$ and some effective divisor $F$ on $X$. It suffices to show that $F=0$. Suppose $F\neq 0$. For very general $t\in B$ and a rational curve $R_t \subset E_t$ which is a deformation of $R$, we see that $[R_t]^\vee$ is a positive multiple of $[E_t]$. Indeed, up to multiples there is only one Hodge class and both $[R_t]^\vee$ and $[E_t]$ pair positively with any K\"ahler class. In particular, $R_t.E_t = \lambda q_{X}(R)$ for some $\lambda > 0$ and for all $t\in B$.  

As deformations of $R$ cover $F$ and $E$, we must have $R.F \geq 0$ and $R.E\geq 0$. We therefore have
\[
\lambda q(R) = E_t.R_t= R.E_0 = mR.E + R.F \geq 0,
\]
a contradiction. This finishes the proof.

(2.b) 
By Theorem~\ref{thm:uniruleddivisor} an irreducible component $C'$ of $R'$ deforms along its Hodge locus  $B$ and the deformations of $C'$ inside 
$$\pi_{\mathcal X}: \mathcal X \to B\subset \Def^\lt(X)$$ 
continue to cover a divisor $E_b$ on $\sX_b$. Let $\pi_{\mathcal H'}:\cH' \to B$ be an irreducible component of the relative Douady space containing the point representing the pair $(X', C')$, which dominates $B$ and is smooth at $(X',C')$ by Theorem~\ref{thm:uniruleddivisor}, and let $\cC' \subset \cH' \times_B \cX$ be the universal subscheme which is smooth along $C'$.   We have the following commutative diagram
$$
\xymatrix{
\cC'\ar[d]_{p} \ar[r]^{ev} & \sX \ar[d]^{\pi_{\mathcal X}}\\
\cH'\ar[r]_{\pi_{\mathcal H'}} & B.
}
$$
Then we claim that 
\begin{equation}\label{eq:2}
\textrm{ the differential } d {ev}: T_{\cC'}\to {ev}^*T_{\cX}\textrm{  is injective along } C'.
\end{equation}
Indeed, we can argue exactly as in \cite[Proof of Lemma 5.1]{Mar13}, because (\ref{eq:norm}) holds as in the smooth case and because by Theorem~\ref{theorem:term} the tangent sheaf ${ev}^*T_{\cX}$ is locally free along $C'$. First, we notice that both 
$T_{\cC'}$ and ${ev}^*T_{\cX}$ have a filtration given respectively by 
$$
0\subset T_p \subset T_{\pi_{\cH'}\circ p} \subset T_{\cC'}
$$
and 
$$
0\subset T_C\subset (ev^*T_{X})_{|C'} \subset ({ev}^*T_{\cX})_{|C'}.
$$
The differential $d {ev}$ is compatible with the filtrations. 
Notice that  the first and third graded summands of both filtrations (restricted to $C'$) are, respectively 
$T_{C'}$ and  $\cO_{C'}^{\oplus (2n-2)}$ and $d {ev}$ obviously induces the identity on them. Therefore, it suffices to prove injectivity on the middle graded summand.
Condition (\ref{eq:norm}) implies
that the evaluation morphism 
$$
 H^0(C', N_{C'/X'})\otimes \cO_{C'}\to N_{C'/X'}
$$
is injective, which in turn implies the injectivity of differential $d {ev}$ on the middle graded summand restricted to $C'$. 

Thanks to (\ref{eq:2}) we obtain that the divisor ${ev}(\cC')$ is reduced over an open subset of $B$ and we are done.
\end{proof} 

The following general result is a converse to one of the statements of Theorem~\ref{thm:exc-def} which we will use in the proof of Theorem \ref{thm:exc}.
\begin{lem}\label{lem:neg}
Let $E\subset X$ be a uniruled prime $\Q$-Cartier divisor inside a projective primitive symplectic variety and chose a ruling. If the general curve of the ruling is either smooth or has two irreducible components meeting transversally in a single point, then $q_X(E)<0$.
\end{lem}
\begin{proof}
By the condition on the general curve $R$ in the ruling of $E$ and the adjunction formula we have that $E\cdot R=\deg E_{|R}=\deg K_R=-2$. If $H$ is a  $\Q$-Cartier effective ample divisor on $X$ such that the pair $(X, \epsilon E+H)$ is still klt, then any MMP for the  pair $(X, \epsilon E)$ directed by $H$ terminates and contracts the strict transform of $E$ by \cite[Corollary~1.4.2]{BCHM10}, compare with the proof of Theorem~{\ref{thm:exc-def}}. Note that having negative intersection number with $E$, the curve $R$ has to be in the restricted base locus of $E$ and as the curves of class $[R]$ cover $E$, we deduce $E=\B_-(E)$. Therefore, the dual $E^\vee$ must be equal to $\lambda R$ for some $\lambda>0$. As a consequence
$$
q_{X}(E)= \lambda E\cdot R= -2\lambda<0.
$$
\end{proof}

%
\section{Rational curves on moduli spaces of sheaves}\label{section rational curves on moduli}
%

We will start with a general lemma on the images of movable divisors under  generically finite, dominant rational maps $p:Y \ratl X$ between normal projective varieties, see Lemma~\ref{lemma rational big movable}. For this purpose, let us first discuss a notion of pushforward under such maps. If $D$ is a Cartier divisor on $Y$, we take a resolution of indeterminacies
\[
\xymatrix{
&Z\ar[dl]_\pi\ar[dr]^q&\\
Y\ar@{-->}[rr]^p&&X\\
}
\]
and define $p_*D$ as follows. Let $\sD \subset Y \times \abs D$ be the universal family over the linear system and denote by $\wt\sD \subset Z \times \abs D$ the strict transform of $\sD$. If $t\in\abs D$ denotes the point corresponding to $D$, then we define $p_*D:= q_* \wt\sD_t$ where $\wt\sD_t$ denotes the fiber over $t$. Note that by construction $p_*D$ is linearly equivalent to $p_*D'$ if $D$ is linearly equivalent to $D'$. Note that $p_*D$ is in general only a Weil divisor. This notion of pushforward has to be digested with care as the following examples show.

\begin{example}\label{example pushforward cremona}
We continue to use the above notation and let $D'$ be a general element of $\abs D$. Assume for simplicity that $D'$ is irreducible and not contracted by $p$. Then $p_*D$ is linearly equivalent to $\deg(p\vert_{D'})\cdot p(D')$ where by $p(D')$ we mean the closure of the image of $D'\vert_U$ where $U$ is the domain of definition of $p$. However, this does not hold for every element in $\abs D$. Let $Y = X = \P^2$ and take $p$ to be the Cremona transformation, i.e. $p([x:y:z])=[yz:xz:xy]$. If $L \subset Y$ is a line not meeting any of the three points of indeterminacy of $p$, then $p_*L = p(L)$ is a conic. Otherwise, the image $p(L)$ is either a line or a point, but in both cases, $p_*L$ is a union of two lines. We also have $p_*L=q_*\pi^*L$ in this case.
\end{example}

\begin{example}\label{example pushforward cubics}
Let $\Sigma:=\Bs\abs D \subset Y$ be the base locus of the linear system of $D$ and suppose that the blow up $\pi:X:=\Bl_\Sigma(Y) \to Y$ is normal. For $p:=\pi^{-1}:Y \ratl X$ we see that $p_*D \neq q_*\pi^*D$, even if $D$ is general in its linear system. (We have $Z=X$ and $q=\id$ here.) The linear system of cubics through $8$ general points in $\P^2$ gives an example where $D$ is even big and nef. More precisely, we take $Y$ to be the blow up of $\P^2$ in the $8$ given points, and we let $D$ be the strict transform of a cubic in $\P^2$ passing through these points so that following the above procedure yields $X$ as a blow up of $Y$ in one point. 
\end{example}

Now we come to the actual purpose of our discussion of pushforward.

\begin{lemma}\label{lemma rational big movable}
Let $p:Y \ratl X$ be a generically finite, dominant rational map between normal projective varieties, suppose that $X$ is $\Q$-factorial, and let $H \subset Y$ be a $\Q$-Cartier divisor. If $H$ is movable or big, then the same holds for $p_*H$.
\end{lemma}
\begin{proof}
Recall that by definition of push forward a general element of $p_*\abs H$ has support equal to the image of a general element of $\abs H$ under $p$. In particular, if $D\subset X$ is a fixed component of $p_*\abs{H}$, then it must be in the image of $p$, in other words, $p$ is well-defined at a point lying over the generic point of $D$. A fortiori this holds for fixed components of $\abs{p_*H}$. As $p$ is generically finite, over a given divisor in $X$ there are at most finitely many divisors in $Y$. In particular, there is an irreducible component of $p^{-1}(D)$ which is a fixed component of $\abs H$. Thus, if $H$ is movable, also $p_*H$ is movable. Note that we have also seen that $\dim p_*\abs{H} \leq \dim\abs{p_*H}$. Applying this to multiples of $H$ shows bigness statement.
\end{proof}

\subsection{Uniruled ample divisors}  
We use the construction introduced in \cite[Lemma 3.9]{PR18} and already used in Proposition \ref{prop:not-triple}. 
Let $S$ be a projective $K3$ surface with an ample divisor $H$. Consider the Mukai vector $v=(0, m\cdot H, 0),\ m\geq 1$ and the moduli space $M_v:=M_v(S,H)$. From the proof of Proposition \ref{prop:not-triple}, we know that $M_v$ is generically finitely dominated by an irreducible holomorphic symplectic manifold $Y$ of $K3^{[n]}$-type and such that $\rho(Y)\geq 2$ by Mukai's isomorphism $v^\perp\cong H^2(Y)$. Here, $Y$ is a crepant resolution of the moduli space $M_{v'}(S,H)$ given by Claim~\ref{claim:stab}, where $v'=(0,m\cdot H,1-\frac{m^2H^2}{2})$ and the dominant map is given in \eqref{eq:PR}.  

We claim that the manifold $Y$ contains infinitely many ample uniruled divisors.
To see this we argue as in \cite[Corollary 4.7]{CMP19}. 
Since $Y$ is projective and has Picard rank at least two, its Picard lattice is indefinite and contains primitive elements of arbitrarily positive BBF square. The same holds for ample classes. Let $h$ be any ample divisor on $Y$ such that $q(h)\geq (2n-2)^2(n-1)$. Recall that if $h$ is a non-zero element of a (nondegenerate) lattice $\Lambda$, the \emph{divisibility} $\div(h)$ of $h$ is the nonnegative integer $t$ such that $h\cdot \Lambda=t\,\Z$. Let $\alpha\in H_2(Y,\mathbb{Z})$ be such that $\alpha^\vee=h/\div(h)$ where $(\cdot)^\vee$ denotes the dual as in Definition~\ref{definition dual class}. In particular, $\alpha$ is a Hodge class. 
We know that $H^2(Y,\Z)=\Lambda_{\rm K3} \oplus\left\langle 2-2n\right\rangle$ where $\Lambda_{\rm K3}$ denotes the unimodular K3-lattice, so the divisibility of $h$ is at most $2n-2$. Therefore, if $q(h)\geq (2n-2)^2(n-1)$, the class $\alpha$ has square at least $n-1$ so that \cite[Proposition~4.6]{CMP19} applies and our claim  follows from \cite[Theorem 4.5]{CMP19}. Notice that the natural pushforward from $H^2(Y,\mathbb{Q})$ to $H^2(M_v,\mathbb{Q})$ is surjective, so that the images of the uniruled divisors on $Y$ span the full Picard lattice of $M_v$, which has rank two for very general $S$ by Mukai's isomorphism (see \cite[Theorem 1.23]{PR18}).

As $M_v$ is $\Q$-factorial by \cite[Theorem A]{KLS06} and \cite[Theorem 1.1]{Pe10}, the push-forward of any ample uniruled divisor $D_Y$ on $Y$ yields a big and movable uniruled divisor $D_v$ on $M_v$ by Lemma~\ref{lemma rational big movable}. Being movable implies that $q_{M_v}(D_v)\geq 0$ which can be seen from the formula \eqref{eq bbf}. By termination of an MMP with scaling for the pair $(M_v,D_v)$, see \cite[Corollary~1.4.2]{BCHM10}, we obtain a birational model $(X,D)$ such that $D$ is nef. But $D_v$ is movable, so the MMP $M_v \ratl \ldots \ratl X$ only consists of flops. In particular, $X$ has $\Q$-factorial terminal singularities by \cite[Proposition 3.37]{KM98} and $q_{M_v}(D_v) = q_X(D) > 0$. Here we used that $D$ is big and nef so that its top self intersection and hence, by Fujiki's formula \cite[Proposition 5.15]{BL18}, its BBF square are positive. By the terminality of $M_v$, see \cite[Remark~1.21]{PR18}, Proposition \ref{prop:terminal}, and Theorem~\ref{theorem:term}, the general curve in the ruling of $D_{M_v}$ does not meet the singular locus of $M_v$. 
By all the above we obtain the following.
\begin{prop}\label{prop:amp-k3}
Let $S$ be a general $K3$ surface and $H$ a primitive ample divisor. Let $m\geq 1$ be a positive integer and consider the moduli space of stable sheaves $M_v$ of Mukai vector $v:=(0,mH,0)$ (and $v$-generic polarization). Then $M_v$ contains infinitely many non-proportional  uniruled divisors with different positive squares, and the general curve in the ruling of such a divisor avoids the singular locus of $M_v$. \qed  
\end{prop}

As observed in \cite[Lemma 3.9]{PR18}, we have an analogous map in the abelian case. In this case, one argues exactly in the same way, replacing $M_u$ and $M_v$ by $K_u$ and $K_v$ and invoking \cite[Corollary 2.3]{MPcorr} instead of \cite[Corollary 4.7]{CMP19}. We then record the following:

\begin{prop}\label{prop:amp-ab}
Let $A$ be a general abelian surface and $H$ a primitive ample divisor. Let $m\geq 1$ be a positive integer and consider the Albanese fiber of the moduli space of stable sheaves $K_v$ of Mukai vector $v:=(0,mH,0)$ (and $v$-generic polarization). Then $K_v$ contains infinitely many non-proportional uniruled divisors with different positive squares and the general curve in the ruling of such a divisor avoids the singular locus of~$K_v$. \qed 
\end{prop}

\begin{proof}[Proof of Theorem \ref{thm:amp}] 
Let $\mathfrak M$ be any moduli space  of polarized primitive symplectic varieties equivalent under locally trivial deformations  to a moduli space  
$M_v(S,\sigma)$ (respectively to $K_v(S,\sigma)$).  By our hypothesis, $\sigma$ is  $v$-generic.
By Corollary \ref{cor:object_to_sheaves}, there is a point in $\mathfrak M$ corresponding to a moduli space of sheaves $X=M_v(S,H)$ (respectively to $X=K_v(S,H)$ in the abelian case) where $v=(0,mH,0)$ and $NS(S)=\mathbb{Z}H$.

 By Propositions \ref{prop:amp-k3} and \ref{prop:amp-ab} the variety $X$ contains infinitely many uniruled divisors $D_v$ of different positive squares. Let $h$ be a positive primitive generator of the image of $\langle D_v\rangle $ under the marking $\phi:H^2(X,\mathbb Z )\to \Lambda$. 
Let $\mathfrak M_h$ be the connected component of $\mathfrak M$ containing  $X$ endowed with the polarization $D_v$. By Theorem \ref{thm:uniruleddivisor} at each point $b$ of $\mathfrak M_h$ the corresponding variety $X_b$ contains a uniruled divisor whose dual class is a positive multiple of the polarization $h$ and we are done. 
\end{proof} 

\begin{rmk}\label{rmk:one}
If $S$ is a $K3$ (or an abelian) surface with $\rho\geq 2$, $H$ an ample divisor on it, and $v=(0,mH,0)$ such that $H$ is $v$-generic, the proof above still yields the existence of a connected component of the moduli space of marked locally trivial deformations of the moduli spaces of sheaves $M_v(S,H)$ (or $K_v(S,H)$). However, as long as we do not control the singularities of $M_v(S,H)$, this variety could possibly have Picard rank one in which case we cannot obtain infinitely many nonproportional uniruled divisors.  
\end{rmk}

\begin{rmk}\label{rmk:monodromy}
A natural follow-up question to Theorem \ref{thm:amp} would consist in determining the cardinality of connected components of $\mathfrak{M}$ which are not covered by the above construction. This would require two ingredients: first one would need to compute precisely the classes of the uniruled divisors, aiming to obtain all (or with finitely many exceptions) isometry orbits of positive divisors. Then one would need to determine their orbit under the monodromy group as these orbits give connected components of $\mathfrak{M}$. The monodromy group is a finite index subgroup of the isometry group of $H^2$ by \cite[Theorem 1.1 (1)]{BL18}, and it is given by all isometries which can be obtained by parallel transport.  A careful study of the Perego--Rapagnetta map could give the first ingredient, but it could still happen that infinitely many of these isometry orbits split into more than one monodromy orbit. Therefore, even if we were able to construct uniruled divisors in all but finitely many isometry orbits, we might still fail to cover countably many monodromy orbits and hence countably many components of $\gothM$. To control the splitting behavior, one would need to determine the monodromy groups for the moduli spaces $M_v$.
\end{rmk}

\subsection{Prime exceptional divisors}\label{ss:exc}

Let $S$ be a general $K3$ surface and a $H$ primitive ample divisor. Let $m\geq 1$ be a positive integer and consider the moduli space of sheaves $M_v$ of Mukai vector $v:=(0,mH,0)$ (and $v$-generic polarization). It has dimension $2g$, where $g$ is the arithmetic genus of $mH$. A general element in this moduli space is an invertible sheaf of degree $g-1$ on a smooth curve $C \in |mH|$. In other words $M_v$ contains as open dense subset the relative Picard variety 
$ \mathcal{J}^{g-1}_{|mH|^\circ}\to |mH|^\circ$ 
fibered over the open subset $|mJ|^\circ\subset |mJ|$ parametrizing smooth curves. For a general choice of $(g-1)$-points $\xi=P_1+P_2+ \dots + P_{g-1}$ on S, there is a pencil $\mathbb P^1_\xi\cong \mathbb P H^0(S, \mathcal O_S(mH)\otimes I_{\xi})$ of curves passing through them.
This pencil comes with a map $\mathbb P^1_\xi \to M_v$ sending a point $t$ in the pencil to the sheaf $\mathcal{O}_{C_t}(P_1+P_2+\dots + P_{g-1})$ in the Picard variety of the curve $C_t$ and thereby gives rise to a section of the lagrangian fibration (\ref{eq:lagr}) restricted to the pencil. This also proves that the rational curve defined in this way is smooth. 

By letting the points vary, we have a $(2g-2)$ dimensional family of rational curves. 
To see that we cover a divisor $D$ in $M_v$ observe that for any smooth $C \in |mH|$ the image of Abel-Jacobi morphism
$C^{(g-1)}\to J^{g-1}(C)$ is a divisor (which is a translate of the theta divisor, see e.g. \cite[p.~338]{GH94}). Therefore, the image of relative Abel-Jabobi morphism over $|mJ|^\circ$
$$
 \mathcal C^{(g-1)}_{|mJ|^\circ} \to \mathcal{J}^{g-1}_{|mJ|^\circ} \subset M_v
$$
is a uniruled divisor which has negative square by Lemma \ref{lem:neg}.

By the terminality of $M_v$ proven in Proposition \ref{prop:terminal} and by Theorem~\ref{theorem:term}, the general curve in the ruling of $D$ does not meet the singular locus of $M_v$. 

By all the above we obtain the following.
\begin{prop}\label{prop:exc-k3}
Let $S$ be a general $K3$ surface and a $H$ primitive ample divisor. Let $m\geq 1$ be a positive integer and consider the moduli space of stable objects $M_v$ of Mukai vector $v:=(0,mH,0)$ (and $v$-generic stability condition). Then $M_v$ contains a prime exceptional divisor and the general curve of its ruling avoids the singular locus of $M_v$. \qed
\end{prop}

\begin{proof}[Proof of Theorem \ref{thm:exc}] 
Let $\mathcal M_0\subset \mathcal M_\Lambda$ be the connected component containing the  moduli space $M_v(S,H)$ (endowed with some marking). For the latter, Proposition \ref{prop:exc-k3} ensures the existence of a prime exceptional divisor $D$. 

Note that by the generality assumption on $S$, the moduli space $M_v(S,H)$ has $\Q$-factorial singularities thanks to \cite[Remark~1.21]{PR18} and therefore up to choice of a marking lies in the same connected component as $M_v(S,\sigma)$. Let us choose a marking $\phi$ on $M_v(S,H)$, put $e:=\phi (D)\in \Lambda$, and consider the Hodge locus $\mathcal D_0 \subset \sM_0$ of $e$. By Theorem \ref{thm:uniruleddivisor}, there exists an non-empty open subset $\mathcal D^o_0$ of $\mathcal D_0$ whose points $(X,f)$ correspond to varieties $X$ all contain a uniruled prime exceptional divisor. On special points of $\mathcal D_0$ the exceptional divisor may not be prime anymore. In this case, one of its components (with the reduced structure) must have negative square and therefore be prime exceptional. We obtain the statement on all the connected components of $\mathcal M_\Lambda$ by the action of the orthogonal group $(X,f)\mapsto (X,g\circ f)$ for $(X,f)\in\mathcal M_0$ and $g\in \O(\Lambda)$.
\end{proof} 

\bibliography{literatur}
\bibliographystyle{alpha}

\end{document}